\documentclass[reqno,english]{amsart}%
\usepackage{amsfonts,amsmath,latexsym,verbatim,amscd,mathrsfs,color,array}
\usepackage[colorlinks=true]{hyperref}
\usepackage{amsmath,amssymb,amsthm,amsfonts,graphicx,color}
\usepackage{amssymb}
\usepackage{pdfsync}
\usepackage{epstopdf}
\usepackage{cite}
\usepackage{graphicx}
\usepackage{amsmath}
\usepackage{amsfonts}%
\providecommand{\U}[1]{\protect\rule{.1in}{.1in}}

\numberwithin{equation}{section}

\newtheorem{theorem}{Theorem}[section]
\newtheorem{corollary}{Corollary}[section]
\newtheorem{lemma}{Lemma}[section]
\newtheorem{proposition}{Proposition}[section]
\newtheorem{remark}{Remark}[section]
\newtheorem{definition}{Definition}[section]

\numberwithin{equation}{section}

\newcommand{\bbr}{\mathbb{R}}

\newcommand{\ve}{\varepsilon}

\newcommand{\bd}{\begin{definition}}
\newcommand{\ed}{\end{definition}}

\newcommand{\br}{\begin{remark}}
\newcommand{\er}{\end{remark}}

\newcommand{\be}{\begin{equation}}
\newcommand{\ee}{\end{equation}}

\newcommand{\bc}{\begin{corollary}}
\newcommand{\ec}{\end{corollary}}

\begin{document}

\title[Schr\"odinger equations]{On some nonlinear Schr\"odinger equations in $\bbr^N$}

\author[J. Wei]{Juncheng Wei}
\address{\noindent Department of Mathematics, University of British Columbia,
Vancouver, B.C., Canada, V6T 1Z2}
\email{jcwei@math.ubc.ca}

\author[Y.Wu]{Yuanze Wu}
\address{\noindent  School of Mathematics, China
University of Mining and Technology, Xuzhou, 221116, P.R. China }
\email{wuyz850306@cumt.edu.cn}

\begin{abstract}
In this paper, we consider the following nonlinear Schr\"{o}dinger equations with the critical Sobolev exponent and mixed nonlinearities:
\begin{eqnarray*}
\left\{\aligned
&-\Delta u+\lambda u=t|u|^{q-2}u+|u|^{2^*-2}u\quad\text{in }\mathbb{R}^N,\\
&u\in H^1(\bbr^N),
\endaligned\right.
\end{eqnarray*}
where $N\geq3$, $t>0$, $\lambda>0$ and $2<q<2^*=\frac{2N}{N-2}$.  Based on our recent study on the normalized solutions of the above equation in \cite{WW21}, we prove that
\begin{enumerate}
\item[$(1)$]\quad the above equation has two positive radial solutions for $N=3$, $2<q<4$ and $t>0$ sufficiently large, which gives a rigorous proof of the numerical conjecture in \cite{DPG13};
\item[$(2)$]\quad there exists $t_q^*>0$ for $2<q\leq4$ such that the above equation has ground-states for $t\geq t_q^*$ in the case of $2<q<4$ and for $t>t_4^*$ in the case of $q=4$ while, the above equation has no ground-states for $0<t<t_q^*$ for all $2<q\leq4$, which, together with the well-known results on ground-states of the above equation, almost completely solve the existence of ground-states to the above equation, except for $N=3$, $q=4$ and $t=t_4^*$.
\end{enumerate}
Moreover, based on the almost completed study on ground-states to the above equation, we introduce a new argument to study the normalized solutions of the above equation to prove that there exists $0<\overline{t}_{a,q}<+\infty$ for $2<q<2+\frac{4}{N}$ such that the above equation has no positive normalized solutions for $t>\overline{t}_{a,q}$ with $\int_{\bbr^N}|u|^2dx=a^2$, which, together with our recent study in \cite{WW21}, gives a completed answer to the open question proposed by Soave in \cite{S20}.  Finally, as applications of our new argument, we also study the following Schr\"{o}dinger equation with a partial confinement:
\begin{eqnarray*}
\left\{\aligned
&-\Delta u+\lambda u+(x_1^2+x_2^2)u=|u|^{p-2}u\quad\text{in }\mathbb{R}^3,\\
&u\in H^1(\bbr^3),\quad \int_{\bbr^3}|u|^2dx=r^2,
\endaligned\right.
\end{eqnarray*}
where $x=(x_1,x_2,x_3)\in\bbr^3$, $\frac{10}{3}<p<6$, $r>0$ is a constant and $(u, \lambda)$ is a pair of unknowns with $\lambda$ being a Lagrange multiplier.  We prove that the above equation has a second positive solution, which is also a mountain-pass solution, for $r>0$ sufficiently small.  This gives a positive answer to the open question proposed by Bellazzini et al. in \cite{BBJV17}.

\vspace{3mm} \noindent{\bf Keywords:} Normalized solution; Ground state; Schr\"odinger equation; Power-type nonlinearity.

\vspace{3mm}\noindent {\bf AMS} Subject Classification 2010: 35B09; 35B33; 35B40; 35J20.%

\end{abstract}

\date{}
\maketitle

\section{Introduction}
In the celebrated paper \cite{GNN81}, the well-known Gidas-Ni-Nirenberg theorem asserts that the positive solution of the following equation,
\begin{eqnarray}\label{eqn0001}
\left\{\aligned
&-\Delta u=f(u)\quad\text{in }\mathbb{R}^N,\\
&u\to0\quad\text{as }|x|\to+\infty,
\endaligned\right.
\end{eqnarray}
must be radially symmetric up to translations under some suitable conditions on the nonlinearities $f(u)$, where $N\geq1$.  Since then, an interesting and important problem is the uniqueness of the positive solution to \eqref{eqn0001}.  Kwong proved such uniqueness result in \cite{K89} for the power-type nonlinearities $f(u)=u^{p-1}-u$ with $2<p<2^*$, where $2^*$ is the critical Sobolev exponent given by $2^*=+\infty$ for $N=1,2$ and $2^*=2N/(N-2)$ for $N\geq3$ (see the earlier papers \cite{C72} for the cubic nonlinearity $f(u)=u^3-u$ and \cite{MS87,PS83,PS86} for general nonlinearities).  The extension of Kwong's result can be found in \cite{M93,PS98,ST00} and so far, to out best knowledge, the most general extension of Kwong's result is due to Serrin and Tang in \cite{ST00}:  The positive solution of \eqref{eqn0001} is unique if there exists $b>0$ such that $\frac{f(u)-u}{u-b}>0$ for $u\not=b$ and the quotient $\frac{f'(u)u-u}{f(u)-u}$ is nonincreasing of $u\in(b, +\infty)$, which is not the case of the mixed nonlinearities $f(u)=\mu u^{q-1}+\nu u^{p-1}-\lambda u$ with $2<q\not=p<2^*$ and $\mu,\nu,\lambda>0$.  In this case, \eqref{eqn0001} reads as
\begin{eqnarray}\label{eqn0002}
\left\{\aligned
&-\Delta u+\lambda u=\mu|u|^{q-2}u+\nu|u|^{p-2}u\quad\text{in }\mathbb{R}^N,\\
&u\to0\quad\text{as }|x|\to+\infty.
\endaligned\right.
\end{eqnarray}
By rescaling, \eqref{eqn0002} is equivalent to
\begin{eqnarray}\label{eqn0003}
\left\{\aligned
&-\Delta u+\lambda u=t |u|^{q-2}u+|u|^{p-2}u\quad\text{in }\mathbb{R}^N,\\
&u\to0\quad\text{as }|x|\to+\infty.
\endaligned\right.
\end{eqnarray}
In an interesting paper \cite{DPG13}, Davila et al. proved that for $N=3$, $2<q<4$, $p<6$ with sufficiently close to $6$ and $t>0$ sufficiently large, \eqref{eqn0003} has three positive radial solutions, which yields a rather striking result that Kwong's uniqueness result is in general not true for the mixed nonlinearities.  Thus, the uniqueness of the positive radial solution of \eqref{eqn0003} (or more general, \eqref{eqn0001}) remains largely open.  It is worth pointing out that the mentioned papers are all devoted to the Sobolev subcritical case for $N\geq3$, that is, $\lim_{u\to+\infty}\frac{f(u)}{u^{2^*}}=0$.

\vskip0.2in

In the Sobolev critical case for $N\geq3$, that is, $\lim_{u\to+\infty}\frac{f(u)}{u^{2^*}}>0$, the well-known Gidas-Ni-Nirenberg theorem still holds, that is, positive solutions must be radially symmetric up to translations.  However, for $N\geq3$, compared to the Sobolev subcritical case (cf. \cite{BL83}), the existence of positive solutions of \eqref{eqn0001} is more complicated in the Sobolev critical case.  For example, for \eqref{eqn0003}, the special case of \eqref{eqn0001}, the existence of positive solutions is established in \cite{AIKN12,ASM12,LLT17,ZZ12,AIIKN19}, which can be summarized as follows:
\begin{theorem}\label{thmn0001}
Let $N\geq3$ and $p=2^*$.  Then \eqref{eqn0003} has a positive radial solution which is also a ground-state, provided that
\begin{enumerate}
\item[$(a)$]\quad $N\geq4$, $2<q<2^*$ and $t>0$;
\item[$(b)$]\quad $N=3$, $4<q<6$ and $t>0$;
\item[$(c)$]\quad $N=3$, $2<q\leq 4$ and $t>0$ sufficiently large.
\end{enumerate}
\end{theorem}
Theorem~\ref{thmn0001} is proved by adapting the classical ideas of Brez\'is and Nirenberg in \cite{BN83}, that is, using the Aubin-Talanti bubbles (cf. \eqref{eq0095}) as test functions to control the energy values so that the $(PS)$ sequences of the associated functional, corresponding to \eqref{eqn0003} with $p=2^*$, are compact at the ground-state level.  This strategy is invalid for $N=3$, $2<q\leq 4$ and $t>0$ not sufficiently large.  Thus, whether \eqref{eqn0003} with $p=2^*$ always has a positive radial solution is not clear.  Note that according to the concentration-compactness principle (cf. \cite{L84}), the only possible way that the $(PS)$ sequences of the associated functional loss the compactness at the ground-state level is that they concentrate at single points and behavior like a Aubin-Talanti bubble under some suitable scalings in passing to the limit.  Thus, by the energy estimates in \cite{AIKN12,ASM12,LLT17,ZZ12,AIIKN19}, it is reasonable to think that \eqref{eqn0003} with $p=2^*$ has no ground-states for $N=3$, $2<q\leq 4$ and $t>0$ not sufficiently large.  On the other hand, the uniqueness of positive radial solutions to \eqref{eqn0003} with $p=2^*$ seems also very complicated.  If $3\leq N\leq6$
and $(N+2)/(N-2)<q<2^*$ then Pucci and Serrin in \cite{PS98} proved that \eqref{eqn0003} with $p=2^*$ has at most one
positive radial solution.  Recently, Akahori et al. in \cite{AIK20,AIKN2021,AIIKN19} and Coles and Gustafson in \cite{CG20} proved that the radial ground-state of \eqref{eqn0003} with $p=2^*$ is unique and nondegenerate for all small $t>0$ when $N\geq5$ and $q\in(2, 2^*)$ or $N=3$ and $q\in(4, 2^*)$; and for all large $t>0$ when $N\geq3$
and $2+4/N<q<2^*$.  However, the uniqueness of positive radial solutions seems not true for \eqref{eqn0003} with $p=2^*$ in general, since it is suggested in \cite{DPG13} by the numerical evidence that \eqref{eqn0003} with $p=2^*$ has two positive radial solutions for $N=3$, $2<q<4$ and $t>0$ sufficiently large.  Moreover, Chen et al. in \cite{CDG16} proved the existence of
arbitrary large number of bubble-tower positive solutions of \eqref{eqn0003} in the slightly supercritical
case when $q<2^*<p=2^*+\ve$ with $\ve>0$ sufficiently small.  We also mention the paper \cite{FG03}, in which the authors proved the existence of positive radial solutions to \eqref{eqn0003} for $2<q<2^*\leq p$ with $t>0$ sufficiently large and \eqref{eqn0003} has no positive solutions for $2<q<2^*< p$ with $t>0$ sufficiently small via ODE's methods.

\vskip0.2in

Inspired by the above facts, we shall explore the existence and nonexistence of positive solutions of \eqref{eqn0003} with $p=2^*$ by studying the existence and nonexistence of ground-states of \eqref{eqn0003} for $N=3$ and $2<q\leq 4$.  We shall also explore the uniqueness of positive solutions of \eqref{eqn0003} with $p=2^*$ by giving a rigorous proof of the numerical conjecture in \cite{DPG13}.

\vskip0.2in

Let us first introduce some necessary notations.  By classical elliptic estimates, for $N\geq3$ and $p=2^*$, \eqref{eqn0003} is equivalent to
\begin{eqnarray}\label{eq0001}
\left\{\aligned
&-\Delta u+\lambda u=t |u|^{q-2}u+|u|^{2^*-2}u\quad\text{in }\mathbb{R}^N,\\
&u\in H^1(\bbr^N),
\endaligned\right.
\end{eqnarray}
where $t>0$, $\lambda>0$ and $2<q<2^*$.  Clearly, by rescaling if necessary, it is sufficiently to consider the case $\lambda=1$ for \eqref{eq0001}.  Let
\begin{eqnarray}\label{eqnewnew0009}
m(t)=\inf_{v\in\mathcal{N}_t}\mathcal{E}_t(v),
\end{eqnarray}
where
\begin{eqnarray}\label{eqn0010}
\mathcal{E}_t(v)=\frac{1}{2}(\|\nabla v\|_2^2+\|v\|_2^2)-\frac{t}{q}\|v\|_q^q-\frac{1}{2^*}\|v\|_{2^*}^{2^*}
\end{eqnarray}
is the corresponding functional of \eqref{eq0001} with $\lambda=1$ and
\begin{eqnarray*}
\mathcal{N}_t=\{v\in H^1(\bbr^N)\backslash\{0\}\mid \mathcal{E}_t'(v)v=0\}
\end{eqnarray*}
is the usual Nehari manifold.  Here, $\|\cdot\|_p$ is the usual norm in the Lebesgue space $L^p(\bbr^N)$.
\begin{definition}
We say that $u$ is a ground-state of \eqref{eq0001} if $u$ is a nontrivial solution of \eqref{eq0001} with $\mathcal{E}_t(u)=m(t)$.
\end{definition}
Now, our main result is the following.
\begin{theorem}\label{thm0002}
Let $\lambda=1$, $N=3$ and $2<q\leq4$.  Then there exists $t_q^*>0$, which may depend on $q$, such that
\begin{enumerate}
\item[$(1)$]\quad \eqref{eq0001} has ground-states for $t\geq t_q^*$ and has no ground-states for $0<t<t_q^*$ in the case of $2<q<4$.
\item[$(2)$]\quad \eqref{eq0001} has ground-states for $t> t_4^*$ and has no ground-states for $0<t<t_4^*$ in the case of $q=4$.
\end{enumerate}
Moreover, if $2<q<4$ then there exists $t_q>0$, which may depend on $q$, such that \eqref{eq0001} has two positive radial solutions $u_{t,1}$ and $u_{t,2}$ for $t>t_q$, where $u_{t,1}$ is a ground-state with $\|u_{t,1}\|_\infty\sim t^{-\frac{1}{q-2}}$ and $u_{t,2}$ is a blow-up solution with
\begin{eqnarray*}
\|u_{t,2}\|_{\infty}\sim\left\{\aligned&t^{\frac{1}{4-q}},\quad 3<q<4,\\
&t\ln t,\quad q=3,\\
&t^{\frac{1}{q-2}},\quad 2<q<3,\endaligned\right.
\end{eqnarray*}
as $t\to+\infty$.
\end{theorem}
\begin{remark}
Theorem~\ref{thm0002}, together with Theorem~\ref{thmn0001}, almost completely solves the existence of ground-states to \eqref{eq0001}, except for $N=3$, $q=4$ and $t=t_4^*$.  Moreover, Theorem~\ref{thm0002} also verifies the numerical conjecture in \cite{DPG13}.
\end{remark}
\vskip0.2in

The proof of Theorem~\ref{thm0002} is based on our very recent study on the normalized solution of \eqref{eq0001} with the additional condition $\|u\|_2^2=a^2$, where $a>0$.  We remark that we shall call $u$ is a fixed-frequency solution of \eqref{eq0001} if the frequency $\lambda$ is fixed, since for the normalized solution of \eqref{eq0001}, the frequency $\lambda$ is a part of unknowns, which appears as a Lagrange multiplier.  Now, let us explain our ideas in proving Theorem~\ref{thm0002}.  Let $\mu>0$, $a>0$ and $(u_\mu, \lambda_\mu)$ be a normalized solution of \eqref{eq0001} for $t=\mu$ with the additional condition $\|u_\mu\|_2^2=a^2$, that is, $(u_\mu, \lambda_\mu)$ is a solution of the following system:
\begin{eqnarray}\label{eqn1001}
\left\{\aligned
&-\Delta u+\lambda u=\mu|u|^{q-2}u+|u|^{2^*-2}u\quad\text{in }\mathbb{R}^N,\\
&u\in H^1(\bbr^N),\quad \|u\|_2^2=a^2,
\endaligned\right.
\end{eqnarray}
then by the Pohozaev identity satisfied by $u_\mu$ (cf. \cite[(4.7)]{WW21}),
\begin{eqnarray}\label{eq0002}
\lambda_{\mu}a^2=\lambda_{\mu}\|u_{\mu}\|_2^2=(1-\gamma_q)\mu\|u_{\mu}\|_q^q>0,
\end{eqnarray}
where $\gamma_q=\frac{N(q-2)}{2q}$.
Let
\begin{eqnarray}\label{eq0003}
v_\mu(x)=\lambda_\mu^{-\frac{N-2}{4}}u_\mu(\lambda_\mu^{-\frac{1}{2}}x),
\end{eqnarray}
then by direct calculations, we know that $v_\mu$ is a fixed-frequency solution of \eqref{eq0001} for $\lambda=1$ and $t=\mu\lambda_\mu^{\frac{q\gamma_q-q}{2}}$.
By \eqref{eq0002}, we also have
\begin{eqnarray*}
\lambda_{\mu}=\frac{(1-\gamma_q)\mu}{a^2}\lambda_\mu^{\frac{q\gamma_q-q}{2}}\|v_\mu\|_q^q.
\end{eqnarray*}
Thus, by letting
\begin{eqnarray}\label{eq0010}
t_\mu=\mu\lambda_\mu^{\frac{q\gamma_q-q}{2}},
\end{eqnarray}
we know that $(v_\mu,t_\mu)$ solves the following system:
\begin{eqnarray}\label{eq0005}
\left\{\aligned
&-\Delta v+v=t |v|^{q-2}v+|v|^{2^*-2}v\quad\text{in }\mathbb{R}^N,\\
&v\in H^1(\bbr^N),\quad t^{\frac{2}{q\gamma_q-q}-1}=\frac{1-\gamma_q}{a^2\mu^{\frac{2}{q-q\gamma_q}}}\|v\|_q^q.
\endaligned\right.
\end{eqnarray}
Clearly, if $(v, t)$ is a solution of the system~\eqref{eq0005}, then by letting
\begin{eqnarray}\label{eq0006}
\lambda_\mu=\bigg(\frac{t}{\mu}\bigg)^{\frac{2}{q\gamma_q-q}}\quad\text{and}\quad u_\mu(x)=\lambda_\mu^{\frac{N-2}{4}}v(\lambda_\mu^{\frac{1}{2}}x),
\end{eqnarray}
$(u_\mu, \lambda_\mu)$ is also a normalized solution of \eqref{eq0001} for $t=\mu$ with the additional condition $\|u_\mu\|_2^2=a^2$, that is $(u_\mu, \lambda_\mu)$ is also a normalized solution of \eqref{eqn1001}.  Thus, by our above observations, normalized solutions of \eqref{eq0001} is equivalent to fixed-frequency solutions of \eqref{eq0001} with another additional condition.  Since we make a detail study on some special normalized solutions of \eqref{eq0001} in \cite{WW21}, we could use these detail estimates to derive Theorem~\ref{thm0002}.

\vskip0.2in

Our observations on the relations between fixed-frequency solutions and normalized solutions of \eqref{eq0001} also bring in some new lights to study the normalized solutions of \eqref{eq0001}.  Indeed, let $v_t$ be a fixed-frequency solution of \eqref{eq0001}, then by the above observations, finding normalized solutions of \eqref{eq0001} is equivalent to finding solutions of the following equation:
\begin{eqnarray}\label{eqn0020}
t^{\frac{2}{q\gamma_q-q}-1}-\frac{1-\gamma_q}{a^2\mu^{\frac{2}{q-q\gamma_q}}}\|v_t\|_q^q=0.
\end{eqnarray}
This is a reduction, which heavily depends on the scaling technique and the Pohozaev identity, since we reduce the solvability of \eqref{eq0001} in $H^1(\bbr^N)$ to the solvability of \eqref{eqn0020} in $\bbr^+$.  Let
\begin{eqnarray*}
\mathcal{A}_\mu(u)=\frac{1}{2}\|\nabla u\|_2^2-\frac{\mu}{q}\|v\|_q^q-\frac{1}{2^*}\|v\|_{2^*}^{2^*}.
\end{eqnarray*}
Then, $\mathcal{A}_\mu|_{S_a}(u)$ is the corresponding functional of \eqref{eqn1001}, where $S_a=\{u\in H^1(\bbr^N)\mid \|u\|_2^2=a^2\}$.
\begin{definition}
We say that $u$ is a normalized ground-state of \eqref{eqn1001} if $u$ is a solution of \eqref{eqn1001} and $\mathcal{A}_\mu(u)\leq\mathcal{A}_\mu(v)$ for any other solutions of \eqref{eqn1001}.
\end{definition}
By \eqref{eq0006}, if $(u_\mu, \lambda_\mu)$ is a solution of \eqref{eqn1001}, then,
\begin{eqnarray*}
\mathcal{A}_\mu(u_\mu)+\frac{\lambda_\mu a^2}{2}=\mathcal{E}_{t_\mu}(v_\mu),
\end{eqnarray*}
where $(v_\mu, t_\mu)$ is a solution of \eqref{eq0005}.  Thus, normalized ground-states of \eqref{eqn1001} must be generated by positive fixed-frequency ground-states of \eqref{eq0001} through the equation~\eqref{eqn0020}.  With these in minds, we can obtain the following results.
\begin{theorem}\label{thm0003}
Let $N\geq3$ and $2<q<2+\frac{4}{N}$.  Then there exist $0<\widehat{t}_{q,a}\leq \overline{t}_{q,a}<+\infty$, which may depend on $q$ and $a$, such that
\eqref{eq0001} has normalized ground-states with the additional condition $\|u\|_2^2=a^2$ for $0<t<\widehat{t}_{q,a}$ and \eqref{eq0001} has no normalized ground-states with the additional condition $\|u\|_2^2=a^2$ for $t>\overline{t}_{q,a}$.
\end{theorem}
\begin{remark}
Theorem~\ref{thm0003}, together with our recent study in \cite{WW21}, gives a completed answer to the open question proposed by Soave in \cite{S20}.
\end{remark}
\vskip0.2in

As an application of our new reduction in finding normalized solutions of \eqref{eq0001}, we shall also consider the following Schr\"odinger equation:
\begin{eqnarray}\label{eq0066}
\left\{\aligned
&-\Delta u+\lambda u+V(x)u=|u|^{p-2}u\quad\text{in }\mathbb{R}^3,\\
&u\in H^1(\bbr^3),\quad \|u\|_2^2=r^2,
\endaligned\right.
\end{eqnarray}
where $x=(x_1,x_2,x_3)\in\bbr^3$, $V(x)=x_1^2+x_2^2$, $\frac{10}{3}<p<6$ and $r>0$ is a constant.  \eqref{eq0066} is studied recently by Bellazzini et al. in \cite{BBJV17}, in which the authors proved that \eqref{eq0066} has a ground-state normalized solution, which is also a local minimizer of the associated functional on the $L^2$-sphere $\|u\|_2^2=r^2$, with a negative Lagrange multiplier $\lambda$ for $r>0$ sufficiently small.  According to the geometry of the associated functional on the $L^2$-sphere $\|u\|_2^2=r^2$, Bellazzini et al. also conjecture in \cite{BBJV17} that \eqref{eq0066} has a second normalized solution, which is also a mountain-pass solution, for $r>0$ sufficiently small.  In this paper, we prove this conjecture by obtaining the following result.
\begin{theorem}\label{thm0004}
Let $\frac{10}{3}<p<6$.  Then for $r>0$ sufficiently small, \eqref{eq0066} has a second positive normalized solution $u_{r,2}$, which is also a mountain-pass solution, with a positive Lagrange multiplier
\begin{eqnarray}\label{eq0099}
\lambda_{r,2}=(1+o_r(1))\bigg[\frac{(6-p)\|w_\infty\|_p^p}{2pr^2}\bigg]^{\frac{2(p-2)}{3p-10}}\to+\infty\quad\text{as }r\to0,
\end{eqnarray}
where $w_\infty$ is the unique (up to translations) positive solution of the following equation:
\begin{eqnarray}\label{eq0096}
\left\{\aligned
&-\Delta w+w=|w|^{p-2}w\quad\text{in }\mathbb{R}^3,\\
&w\in H^1(\bbr^3).
\endaligned\right.
\end{eqnarray}
Moreover,
\begin{eqnarray}\label{eq0098}
w_{r}(x)=\lambda_{r,2}^{-\frac{1}{p-2}}u_{r,2}(\lambda_{r,2}^{-\frac12}x)=w_\infty+o_r(1)\quad\text{in }H^1(\bbr^3)\quad\text{as }r\to0.
\end{eqnarray}
\end{theorem}

To prove Theorem~\ref{thm0004}, we apply our new reduction argument to \eqref{eq0066} by reducing finding normalized solutions of \eqref{eq0066} to finding solutions of the following equation:
\begin{eqnarray}\label{eqn0030}
f(r,t):=r^2-t^{\frac{10-3p}{2(p-2)}}(\frac{6-p}{2p}\|w_t\|_p^p-2t^{-2}\int_{\bbr^3}V(x)w_t^2dx),
\end{eqnarray}
where $w_t$ is a positive ground-state of the following equation:
\begin{eqnarray*}
\left\{\aligned
&-\Delta w+w+t^{-2}V(x)w=|w|^{p-2}w\quad\text{in }\mathbb{R}^3,\\
&w\in H^1(\bbr^3).
\endaligned\right.
\end{eqnarray*}
By the uniqueness and nondegeneracy of $w_\infty$, we prove that the curve $w_t$ is continuous for $t>0$ sufficiently large in a suitable space.  Thus, \eqref{eqn0030} can be solved easily by the continuation method.  We believe this method will be helpful in studying normalized solutions of other elliptic equations.

\vskip0.2in

\noindent{\bf\large Notations.} Throughout this paper, $C$ and $C'$ are indiscriminately used to denote various absolutely positive constants.  $a\sim b$ means that $C'b\leq a\leq Cb$ and $a\lesssim b$ means that $a\leq Cb$.

\section{blow-up solutions for $N=3$ and $2<q<4$}
It is well known that the Aubin-Talanti babbles,
\begin{eqnarray}\label{eq0095}
U_\ve(x)=[N(N-2)]^{\frac{N-2}{4}}\bigg(\frac{\ve}{\ve^2+|x|^2}\bigg)^{\frac{N-2}{2}},
\end{eqnarray}
is the only solutions to the following equation:
\begin{eqnarray*}
\left\{\aligned&-\Delta u=u^{2^*-1}\quad\text{in }\bbr^N,\\
&u(0)=\max_{x\in\bbr^N}u(x),\\
&u(x)>0\quad\text{in }\bbr^N,\\
&u(x)\to0\quad\text{as }|x|\to+\infty.
\endaligned\right.
\end{eqnarray*}
By \cite[Theorem~1.2]{WW21}, for $\mu>0$ sufficiently small, \eqref{eqn1001} has a positive radial solution $\widetilde{u}_\mu$ with the Lagrange multiplier $\widetilde{\lambda}_\mu>0$ such that $\ve_{\mu}^{\frac{1}{2}}\widetilde{u}_\mu(\ve_{\mu} x)\to U_{\ve_0}$ strongly in $D^{1,2}(\bbr^3)$ for some $\ve_0>0$ as $\mu\to0$ up to a subsequence,
where $U_{\ve_0}$ is given by \eqref{eq0095} and $\ve_{\mu}$ satisfies
\begin{eqnarray}\label{eq0007}
\mu\sim\left\{\aligned
&\ve_{\mu}^{\frac{q}{2}-1},\quad 3<q<6,\\
&\frac{\ve_{\mu}^{\frac{1}{2}}}{\ln(\frac{1}{\ve_\mu})},\quad q=3,\\
&\ve_\mu^{5-\frac{3q}{2}},\quad 2<q<3.\endaligned\right.
\end{eqnarray}
Moreover, by \cite[Lemma~4.1]{WW21}, we have
\begin{eqnarray}\label{eq0012}
1\sim\left\{\aligned &\frac{\mu\sigma_\mu^{\frac{6-q}{2}}}{\widetilde{\lambda}_{\mu}},\quad 3<q<6,\\
&\frac{\mu \sigma_\mu^{\frac{3}{2}}}{\widetilde{\lambda}_{\mu}}\ln\bigg(\frac{1}{\sqrt{\widetilde{\lambda}_{\mu}}\sigma_\mu}\bigg),\quad q=3,\\
&\frac{\mu \sigma_\mu^{\frac{q}{2}}}{\widetilde{\lambda}_{\mu}^{\frac{5-q}{2}}},\quad 2<q<3.\endaligned\right.
\end{eqnarray}
On the other hand, in the proof of \cite[Proposition~4.2]{WW21}, we also show that
\begin{eqnarray}\label{eq0060}
\sigma_\mu\sim\ve_\mu\quad\text{as}\quad \mu\to0.
\end{eqnarray}
\begin{proposition}\label{propn0001}
Let $\lambda=1$, $N=3$ and $2<q<4$.  Then there exists $t_q>0$, which may depend on $q$, such that \eqref{eq0001} has two positive radial solutions $u_{t,1}$ and $u_{t,2}$ for $t>t_q$, where $u_{t,1}$ is a ground-state with $\|u_{t,1}\|_\infty\sim t^{-\frac{1}{q-2}}$ and $u_{t,2}$ is a blow-up solution with
\begin{eqnarray*}
\|u_{\mu,2}\|_{\infty}\sim\left\{\aligned&t^{\frac{1}{4-q}},\quad 3<q<4,\\
&t\ln t,\quad q=3,\\
&t^{\frac{1}{q-2}},\quad 2<q<3,\endaligned\right.
\end{eqnarray*}
as $t\to+\infty$.
\end{proposition}
\begin{proof}
By \eqref{eq0003} and \eqref{eq0010}, $(\widetilde{v}_\mu, \widetilde{t}_\mu)$ is a solution of \eqref{eq0005}.  In particular, $\widetilde{v}_\mu$ is a solution of \eqref{eq0001} for $\lambda=1$ and $t=\widetilde{t}_\mu=\mu\widetilde{\lambda}_\mu^{\frac{q\gamma_q-q}{2}}$.  By the well-known Gidas-Ni-Nirenberg theorem \cite{GNN81}, $\widetilde{v}_{\mu}$ is radial and decreasing for $r=|x|$  up to translations.  Thus, without loss of generality, we may assume that $\widetilde{v}_{\mu}(0)=\max_{x\in\bbr^N}\widetilde{v}_{\mu}$.  Recall that $\ve_{\mu}^{\frac{1}{2}}\widetilde{u}_\mu(\ve_{\mu} x)\to U_{\ve_0}$ strongly in $D^{1,2}(\bbr^3)$ for some $\ve_0>0$ as $\mu\to0$ up to a subsequence, by the classical elliptic regularity and the Sobolev embedding theorem, $\ve_{\mu}^{\frac{1}{2}}\widetilde{u}_\mu(\ve_{\mu} x)\to U_{\ve_0}$ strongly in $C_{loc}^{1,\alpha}(\bbr^3)$ for some $\alpha\in(0, 1)$ as $\mu\to0$ up to a subsequence.  In particular, $\ve_{\mu}^{\frac{1}{2}}\widetilde{u}_\mu(0)\to U_{\ve_0}(0)$ as $\mu\to0$ up to a subsequence.  Thus, by \eqref{eq0095},
\begin{eqnarray}\label{eq0008}
\widetilde{v}_{\mu}(0)=\widetilde{\lambda}_\mu^{-\frac{1}{4}}\widetilde{u}_\mu(0)\sim\widetilde{\lambda}_\mu^{-\frac{1}{4}}\ve_\mu^{-\frac{1}{2}}\quad\text{as }\mu\to0\text{ up to a subsequence}.
\end{eqnarray}
In the following, let us estimates $\widetilde{v}_{\mu}(0)$ and $\widetilde{t}_\mu$ as $\mu\to0$.  We begin with the estimate of $\widetilde{t}_\mu$.  We first consider the case $2<q<3$. In this case, by \eqref{eq0007}, \eqref{eq0012} and \eqref{eq0060}, $\widetilde{\lambda}_\mu\sim\ve_\mu^2,$
which, together with \eqref{eq0010}, implies
\begin{eqnarray*}
\widetilde{t}_\mu\sim\ve_\mu^{\frac{10-3q}{2}}(\ve_\mu^2)^{\frac{q-6}{4}}=\ve_\mu^{2-q}\to+\infty\quad\text{as }\mu\to0.
\end{eqnarray*}
For $q=3$, by \eqref{eq0007}, \eqref{eq0012} and \eqref{eq0060},
\begin{eqnarray*}
\widetilde{\lambda}_\mu\sim\ve_\mu^2\frac{\ln(\frac{1}{\sqrt{\widetilde{\lambda}_\mu}\ve_\mu})}{\ln(\frac{1}{\ve_\mu})}\gtrsim\ve_\mu^2.
\end{eqnarray*}
It follows that
\begin{eqnarray*}
\ln(\frac{1}{\ve_\mu})\lesssim\ln(\frac{1}{\sqrt{\widetilde{\lambda}_\mu}\ve_\mu})\lesssim\ln(\frac{1}{\ve_\mu}).
\end{eqnarray*}
Thus, we also have $\widetilde{\lambda}_\mu\sim\ve_\mu^2$ for $q=3$.
By \eqref{eq0010} and \eqref{eq0007},
\begin{eqnarray*}
\widetilde{t}_\mu\sim\ve_\mu^{\frac{1}{2}}\frac{1}{\ln(\frac{1}{\ve_\mu})}(\ve_\mu^2)^{-\frac{3}{4}}=\ve_\mu^{-1}\frac{1}{\ln(\frac{1}{\ve_\mu})}\to+\infty\quad\text{as }\mu\to0.
\end{eqnarray*}
For $3<q<4$, by \eqref{eq0007}, \eqref{eq0012} and \eqref{eq0060}, $\widetilde{\lambda}_\mu\sim\ve_\mu^2.$
Now, by \eqref{eq0010},
\begin{eqnarray*}
\widetilde{t}_\mu\sim\ve_\mu^{\frac{q-2}{2}}(\ve_\mu^2)^{\frac{q-6}{4}}=\ve_\mu^{q-4}\to+\infty\quad\text{as }\mu\to0.
\end{eqnarray*}
Thus, for all $2<q<4$, we always have
\begin{eqnarray}\label{eq0014}
\widetilde{\lambda}_\mu\sim\ve_\mu^2\quad\text{and}\quad\widetilde{t}_\mu\sim\left\{\aligned&\ve_\mu^{q-4},\quad 3<q<4,\\
&\ve_\mu^{-1}\frac{1}{\ln(\frac{1}{\ve_\mu})},\quad q=3,\\
&\ve_\mu^{2-q},\quad 2<q<3,
\endaligned\right.
\end{eqnarray}
as $\mu\to0$.
Now, by \eqref{eq0007}--\eqref{eq0060} and \eqref{eq0008}, we have
\begin{eqnarray*}
\widetilde{v}_{\mu}(0)\sim\left\{\aligned&\mu^{-\frac{2}{q-2}},\quad 3<q<4,\\
&\bigg(\frac{1}{\mu|\ln\mu|}\bigg)^2,\quad q=3,\\
&\mu^{-\frac{2}{10-3q}},\quad 2<q<3.\endaligned\right.
\end{eqnarray*}
It follows from \eqref{eq0007} and \eqref{eq0014} that
\begin{eqnarray*}
\widetilde{v}_{\mu}(0)\sim\left\{\aligned&\widetilde{t}_\mu^{\frac{1}{4-q}},\quad 3<q<4,\\
&\widetilde{t}_\mu\ln \widetilde{t}_\mu,\quad q=3,\\
&\widetilde{t}_\mu^{\frac{1}{q-2}},\quad 2<q<3.\endaligned\right.
\end{eqnarray*}
Thus, by \eqref{eq0014}, $\widetilde{v}_{\mu}$ is a blow-up solution of \eqref{eq0001} for $N=3$, $\lambda=1$, $2<q<4$ and $t=\widetilde{t}_\mu$.  Note that by \cite[Theorem~2.2]{MM21}, the ground-states of \eqref{eq0001} for $\lambda=1$, say $\overline{v}_{t}$, satisfies $\|\overline{v}_{t}\|_\infty\sim t^{-\frac{1}{q-2}}$ as $t\to+\infty$.  For $\mu>0$ sufficiently small, $\widetilde{v}_{\mu}$ is a second positive radial solution of \eqref{eq0001} with $N=3$, $\lambda=1$, $2<q<4$ and $t>0$ sufficiently large.
\end{proof}

\begin{remark}
Let $\widetilde{v}_\mu$ be given in the proof of Proposition~\ref{propn0001} and define
\begin{eqnarray*}\label{eq0018}
\widetilde{w}_{\mu}(x)=\widetilde{t}_\mu^{\frac{1}{q-2}}\widetilde{v}_\mu(x),
\end{eqnarray*}
then $\widetilde{w}_{\mu}$ satisfies the following equation:
\begin{eqnarray}\label{eq0024}
\left\{\aligned
&-\Delta w+w=|w|^{q-2}w+\widetilde{t}_\mu^{-\frac{2^*-2}{q-2}}|w|^{2^*-2}w\quad\text{in }\mathbb{R}^N,\\
&v\in H^1(\bbr^N),
\endaligned\right.
\end{eqnarray}
where $t=\widetilde{t}_\mu$ is also given in the proof of Proposition~\ref{propn0001}.  By similar arguments as that used for \cite[Lemma~5.3]{DPG13}, \eqref{eq0024} has a unique bounded positive radial solution for $t>0$ sufficiently large.
However, by \eqref{eq0010} and \eqref{eq0008},
\begin{eqnarray}\label{eq0011}
\widetilde{w}_{\mu}(0)\sim\mu^{\frac{1}{q-2}}\widetilde{\lambda}_\mu^{-\frac{1}{q-2}}\ve_\mu^{-\frac{1}{2}}\quad\text{as }\mu\to0.
\end{eqnarray}
By \eqref{eq0007}, \eqref{eq0014} and \eqref{eq0011},
\begin{eqnarray*}
\widetilde{w}_{\mu}(0)\sim\left\{\aligned&\ve_\mu^{-\frac{2}{q-2}},\quad 2<q<3,\\
&\ve_\mu^{-2}\frac{1}{\ln(\frac{1}{\ve_\mu})},\quad q=3,\\
&\ve_\mu^{-\frac{2}{q-2}},\quad 3<q<4.
\endaligned\right.
\end{eqnarray*}
Thus, $\widetilde{w}_{\mu}$ is also a blow-up solution of \eqref{eq0024} as $\widetilde{t}_\mu\to+\infty$.
\end{remark}

\section{Ground-states for $N=3$ and $2<q\leq 4$}
The associated fibering map of \eqref{eqn0010} for every $v\not=0$ in $H^1(\bbr^3)$ is given by
\begin{eqnarray}\label{eq0075}
E(s)=\frac{s^2}{2}(\|\nabla v\|_2^2+\|v\|_2^2)-\frac{ts^q}{q}\|v\|_q^q-\frac{s^6}{6}\|v\|_6^6.
\end{eqnarray}
Since $q>2$, it is standard to show that for every $v\not=0$ in $H^1(\bbr^3)$, there exists a unique $s_0>0$ such that $E(s)$ is strictly increasing for $0<s<s_0$ and strictly decreasing for $s>s_0$.
\begin{lemma}\label{lem0001}
Let $N=3$, $\lambda=1$ and $2<q\leq4$.  Then $m(t)=\frac{1}{3}S^{\frac{3}{2}}$ for $t>0$ sufficiently small, where $m(t)$ is given by \eqref{eqnewnew0009}.
\end{lemma}
\begin{proof}
We argue in the contrary by supposing that there exists $t_n\to0$ as $n\to\infty$ such that $m(t_n)<\frac{1}{3}S^{\frac{3}{2}}$.  Then, it is standard to show (cf.\cite{ASM12}) that $m(t_n)$ is attained by a positive and radial function, which is also a solution of \eqref{eq0001} with $\lambda=1$, $N=3$ and $t=t_n$.  We denote this solution by $v_{t_n}$.  Since $t_n\to0$ as $n\to\infty$, it is also standard to show that
\begin{eqnarray}\label{eq0065}
\|\nabla v_{t_n}\|_2^2=\|v_{t_n}\|_6^6+o_n(1)=S^{\frac{3}{2}}+o_n(1)\quad\text{as }n\to\infty.
\end{eqnarray}
Thus, $\{v_{t_n}\}$ is a minimizing sequence of the Sobolev inequality.  By Lions' result (cf. \cite[Theorem~1.41]{W96}), up to a subsequence, there exists $\sigma_n>0$ such that for some $\ve_*>0$,
\begin{eqnarray*}
w_{t_n}(x)=\sigma_n^{\frac{1}{2}}v_{t_n}(\sigma_n x)\to U_{\varepsilon_*}\text{ strongly in }D^{1,2}(\bbr^3)\text{ as }n\to\infty.
\end{eqnarray*}
Clearly, by direct computations, we know that $w_{t_n}$ satisfies the following equation:
\begin{eqnarray}\label{eq0021}
-\Delta w_{t_n}+\sigma_n^2w_{t_n}=t_n\sigma_n^{3-\frac{q}{2}}w_{t_n}^{q-1}+w_{t_n}^{5}\quad\text{in }\mathbb{R}^3.
\end{eqnarray}
Since $v_{t_n}$ is positive and radial, $w_{t_n}$ is also positive and radial.  Thus, by the boundedness of $\{w_{t_n}\}$ in $D^{1,2}(\bbr^3)$, the Sobolev embedding theorem and Struss's radial lemma (cf. \cite[Lemma
A.2]{BL83}),
\begin{eqnarray*}
w_{t_n}\lesssim r^{-\frac{1}{2}}\quad\text{for all }r\geq1\text{ uniformly as }n\to\infty.
\end{eqnarray*}
On the other hand, since $w_{t_n}\to U_{\varepsilon_*}$ strongly in $D^{1,2}(\bbr^3)$ as $n\to\infty$, by applying the Moser iteration in a standard way and using the Sobolev embedding theorem, we know that $w_{t_n}\to U_{\varepsilon_*}$ strongly in $C_{loc}^{1,\alpha}(\bbr^3)$ as $n\to\infty$ for some $\alpha\in(0, 1)$.  Thus,
\begin{eqnarray*}
w_{t_n}\lesssim (1+r)^{-\frac{1}{2}}\quad\text{for all }r\geq0\text{ uniformly as }n\to\infty.
\end{eqnarray*}
Now, we can adapt the ODE's argument in \cite{AP86,GS03,KP89} as that in the proof of \cite[Lemma~4.1]{WW21} to obtain
\begin{eqnarray}\label{eq0020}
w_{t_n}\lesssim\frac{1}{(1+r^2)^{\frac12}}\quad\text{for all }r\geq0\text{ uniformly as }n\to\infty.
\end{eqnarray}
On the other hand, since $N=3$, it is easy to check that $r^{-1}e^{-\sigma_nr}$ is a subsolution of $-\Delta u+\sigma_n^2u=0$ for $r\geq1$.  Thus, by the fact that $w_{t_n}\to U_{\varepsilon_*}$ strongly in $C_{loc}^{1,\alpha}(\bbr^3)$ as $n\to\infty$ for some $\alpha\in(0, 1)$, we can use the maximum principle in a standard way to show that
\begin{eqnarray*}
w_{t_n}\gtrsim r^{-1}e^{-\sigma_nr}\quad \text{for }r\geq1\text{ uniformly as }n\to\infty.
\end{eqnarray*}
It follows that
\begin{eqnarray}\label{eq0061}
\|w_{t_n}\|_q^q\gtrsim\int_1^{\frac{1}{\sigma_n}}r^{2-q}e^{-q\sigma_nr}dr\sim\left\{\aligned&\sigma_n^{q-3},\quad2\leq q<3,\\
&|\ln\sigma_n|,\quad q=3,\\
&1,\quad 3<q<6.
\endaligned\right.
\end{eqnarray}
Since $t_n\to0$ as $n\to\infty$, by \eqref{eq0020}, for $r\gtrsim\bigg(\frac{1}{\sigma_n}\bigg)^{\frac12}$, \eqref{eq0021} reads as
\begin{eqnarray*}
-\Delta w_{t_n}+\frac{1}{4}\sigma_n^2w_{t_n}\leq0\quad\text{in }\mathbb{R}^3.
\end{eqnarray*}
Thus, by \eqref{eq0020}, we can use the maximum principle in a standard way again to obtain
\begin{eqnarray*}
w_{t_n}\lesssim r^{-1}e^{-\frac{\sigma_n}{4}r}\quad \text{for }r\gtrsim\bigg(\frac{1}{\sigma_n}\bigg)^{\frac12}\text{ uniformly as }n\to\infty.
\end{eqnarray*}
On the other hand, since $\|w_{t_n}\|_6^6=\|v_{t_n}\|_6^6=S^{\frac{3}{2}}+o_n(1)$, by \eqref{eq0040} and the H\"older inequality,
\begin{eqnarray*}
\sigma_n^2\|w_{t_n}\|_2^2\lesssim t_n\sigma_n^{3-\frac{q}{2}}\|w_{t_n}\|_q^q\lesssim t_n\sigma_n^{3-\frac{q}{2}}\|w_{t_n}\|_2^{\frac{6-q}{2}},
\end{eqnarray*}
which implies
\begin{eqnarray*}
\sigma_n\|w_{t_n}\|_2\lesssim t_n^{\frac{2}{q-2}}.
\end{eqnarray*}
Since $w_{t_n}\to U_{\ve_*}$ strongly in $D^{1,2}(\bbr^3)$ as $n\to\infty$ and $U_{\ve_*}\not\in L^2(\bbr^3)$, by the Fatou lemma,
\begin{eqnarray*}
\liminf_{n\to\infty}\|w_{t_n}\|_2=+\infty.
\end{eqnarray*}
Thus, by $t_n\to0$ as $n\to\infty$, we have $\sigma_n\to0$ as $n\to\infty$.
It follows from \eqref{eq0020} once more that
\begin{eqnarray}\label{eq0062}
\|w_t\|_q^q\lesssim1+\int_1^{\frac{1}{\sigma_n}}r^{2-q}dr+\int_{(\frac{1}{\sigma_n})^{\frac{1}{2}}}^{+\infty}r^{2-q}e^{-\frac{q}{4}\sigma_nr}dr\sim\left\{\aligned&\sigma_n^{q-3},\quad2\leq q<3,\\
&|\ln\sigma_n|,\quad q=3,\\
&1,\quad 3<q<6.
\endaligned\right.
\end{eqnarray}
Thus, by \eqref{eq0061} and \eqref{eq0062}, we have
\begin{eqnarray}\label{eq0063}
\|w_{t_n}\|_q^q\sim\left\{\aligned&\sigma_n^{q-3},\quad2\leq q<3,\\
&|\ln\sigma_n|,\quad q=3,\\
&1,\quad 3<q<6.
\endaligned\right.
\end{eqnarray}
Note that as that of \eqref{eq0002}, by the Pohozaev identity, we have
\begin{eqnarray}\label{eq0040}
\sigma_n^2\|w_{t_n}\|_2^2=(1-\gamma_q)t_n\sigma_n^{3-\frac{q}{2}}\|w_{t_n}\|_q^q.
\end{eqnarray}
Thus, by \eqref{eq0063},
\begin{eqnarray*}
\sigma_n\sim \left\{\aligned&t_n\sigma_n^{\frac{q}{2}},\quad2< q<3,\\
&t_n\sigma_n^{\frac{3}{2}}|\ln\sigma_n|,\quad q=3,\\
&t_n\sigma_n^{3-\frac{q}{2}},\quad 3<q<6,
\endaligned\right.
\end{eqnarray*}
which implies
\begin{eqnarray}\label{eq0041}
t_n\sim \left\{\aligned&\sigma_n^{\frac{2-q}{2}},\quad2< q<3,\\
&\sigma_n^{-\frac{1}{2}}\frac{1}{|\ln\sigma_n|},\quad q=3,\\
&\sigma_n^{\frac{q-4}{2}},\quad 3<q<6.
\endaligned\right.
\end{eqnarray}
\eqref{eq0041} contradicts the facts that $t_n,\sigma_n\to0$ as $n\to\infty$ for $2<q\leq4$.  It follows that $m(t)\geq\frac{1}{3}S^{\frac{3}{2}}$ for $t>0$ sufficiently small in the case of $2<q\leq4$.  On the other hand, since $m(t)$ is the minimum of $\mathcal{E}_t(v)$ on the Nehari manifold $\mathcal{N}_{t}$, it is standard (cf. \cite[Lemma~3.3]{WW21}) to use the fibering maps~\eqref{eq0075} to show that $m(t)$ is nonincreasing for $t>0$.  Note that it is well known that $m(0)=\frac{1}{3}S^{\frac{3}{2}}$, thus, $m(t)\leq\frac{1}{3}S^{\frac{3}{2}}$ for all $t>0$.  It follows that $m(t)=\frac{1}{3}S^{\frac{3}{2}}$ for $t>0$ sufficiently small in the case of $2<q\leq4$.
\end{proof}

Let
\begin{eqnarray}\label{eqn0090}
t_q^*=\sup\{t>0\mid m_t=\frac{1}{3}S^{\frac{3}{2}}\}.
\end{eqnarray}
Then by Lemma~\ref{lem0001}, $t_q^*>0$ for $2<q\leq4$.  Since it is well known (cf. \cite{ASM12}) that $m(t)<\frac{1}{3}S^{\frac{3}{2}}$ for $t>0$ sufficiently large in the case of $2<q\leq4$, we have $0<t_q^*<+\infty$ for all $2<q\leq4$.  Since $m(t)<\frac{1}{3}S^{\frac{3}{2}}$ for $t>t_q^*$, it is standard (cf. \cite{ASM12}) to show that $m(t)$ is attained for $t>t_q^*$.  Let $v_t$ be a ground-state of \eqref{eq0001}, which is radial and positive for $t>t_q^*$ in the case of $2<q<4$.  Then, we have the following.
\begin{proposition}\label{prop0001}
Let $N=3$, $\lambda=1$ and $2<q<4$.  Then, $\|v_t\|_q^q\sim1$ as $t\to t_q^*$.
\end{proposition}
\begin{proof}
The conclusion $\|v_t\|_q^q\lesssim1$ as $t\to t_q^*$ is standard so we omit it.  For the conclusion $\|v_t\|_q^q\gtrsim1$ as $t\to t_q^*$ , we argue in the contrary.  Then there exists $t_n\to t_q^*$ as $n\to\infty$ such that $\|v_{t_n}\|_q^q\to0$ as $n\to \infty$.  Similar to that of \eqref{eq0065}, we also have
\begin{eqnarray*}
\|\nabla v_{t_n}\|_2^2=\|v_{t_n}\|_6^6+o_n(1)=S^{\frac{3}{2}}+o_n(1)\quad\text{as }n\to\infty.
\end{eqnarray*}
Thus, $\{v_{t_n}\}$ is a minimizing sequence of the Sobolev inequality.  By Lions' result (cf. \cite[Theorem~1.41]{W96}), up to a subsequence, there exists $\sigma'_n>0$ such that for some $\ve_*>0$,
\begin{eqnarray*}
w_n(x)=(\sigma'_n)^{\frac{1}{2}}v_{t_n}(\sigma'_n x)\to U_{\varepsilon_*}\text{ strongly in }D^{1,2}(\bbr^3)\text{ as }n\to \infty.
\end{eqnarray*}
Now, repeating the arguments for \eqref{eq0041}, we will arrive at
\begin{eqnarray*}
t_q^*\sim \left\{\aligned&(\sigma_n')^{\frac{2-q}{2}},\quad2< q<3,\\
&(\sigma_n')^{-\frac{1}{2}}\frac{1}{|\ln\sigma'_n|},\quad q=3,\\
&(\sigma'_n)^{\frac{q-4}{2}},\quad 3<q<4.
\endaligned\right.
\end{eqnarray*}
This is impossible since $\sigma'_n\to0$ as $n\to \infty$ by similar arguments as that used for $\sigma_n$ in the proof of Lemma~\ref{lem0001}.  Thus, we must have $\|v_t\|_q^q\gtrsim1$ as $t\to t_q^*$.
\end{proof}

Now, we are arriving at the following.
\begin{proposition}\label{propn0002}
Let $\lambda=1$, $N=3$ and $2<q\leq4$.  Then
\begin{enumerate}
\item[$(1)$]\quad \eqref{eq0001} has ground-states for $t\geq t_q^*$ and has no ground-states for $0<t<t_q^*$ in the case of $2<q<4$.
\item[$(2)$]\quad \eqref{eq0001} has ground-states for $t> t_4^*$ and has no ground-states for $0<t<t_4^*$ in the case of $q=4$.
\end{enumerate}
Here, $t_q^*$ is given by \eqref{eqn0090}.
\end{proposition}
\begin{proof}
We first prove that there is no ground-states of \eqref{eq0001} for $0<t<t_q^*$ in the case of $2<q\leq 4$.  Suppose the contrary that \eqref{eq0001} has a ground-state for some $0<t<t_q^*$ in the case of $2<q\leq 4$.  Then $m(t)$ is attained.  Now, by use the fibering maps~\eqref{eq0075} in a standard way (cf. \cite[Lemma~3.3]{WW21}), we have $m(t')<m(t)$ for all $t'>t$.  It follows that
$m(t')<\frac{1}{3}S^{\frac{3}{2}}$ for all $t'>t$, which contradicts the definition of $t_q^*$ given by \eqref{eqn0090}.  Thus, there is no ground-states of \eqref{eq0001} for $0<t<t_q^*$ in the case of $2<q\leq 4$.  It remains to prove that \eqref{eq0001} has a ground-state for $t=t_q^*$ in the case of $2<q<4$, which is equivalent to prove that $m(t_q^*)$ is attained for $2<q<4$.  Let $v_t$ be a ground-state of \eqref{eq0001}, which is radial and positive for $t>t_q^*$ in the case of $2<q<4$ such that $t\to t_q^*$.  By Proposition~\ref{prop0001}, $\|v_t\|_q^q\gtrsim1$ as $t\to t_q^*$.  Since $v_t$ is radial, it is standard to show that $v_t\to v_{t_q^*}\not=0$ strongly in $H^1(\bbr^3)$ as $t\to t_q^*$ up to a subsequence.  Thus, $m(t_q^*)$ is attained by $v_{t_q^*}$, which is also a ground-state of \eqref{eq0001} for $t=t_q^*$ in the case of $2<q<4$.
\end{proof}

\begin{remark}
Upon to Theorem~\ref{thm0002}, the existence of ground-states of \eqref{eq0001} is almost completely solved, except for $N=3$, $q=4$ and $t=t_4^*$.  In this case, we believe that there is no ground-states of \eqref{eq0001}.  Indeed, let $\mu>0$, $a>0$ and $(u_\mu, \lambda_\mu)$ be a normalized solution of \eqref{eqn1001}, then by \eqref{eq0003} and \eqref{eq0010}, $\widetilde{v}_\mu$ is a solution of \eqref{eq0001} with $\lambda=1$ and $t=\widetilde{t}_\mu=\mu\widetilde{\lambda}_\mu^{\frac{q\gamma_q-q}{2}}$.
By \eqref{eq0007} and \eqref{eq0014},
\begin{eqnarray*}
\widetilde{t}_\mu\sim\mu\ve_\mu^2\sim\mu^{\frac{2(q-4)}{q-2}}\quad\text{as }\mu\to0\text{ for }4\leq q<6.
\end{eqnarray*}
Thus, $\widetilde{t}_\mu\to0$ as $\mu\to0$ for $4<q<6$ and $\widetilde{t}_\mu\sim1$ as $\mu\to0$ for $q=4$.  Note that $\widetilde{v}_\mu$, generated by $\widetilde{u}_\mu$ though \eqref{eq0003}, is a solution of \eqref{eq0001} for $t=\widetilde{t}_\mu$ and by \cite[Theorem~1.2]{WW21},
\begin{eqnarray*}
\|\nabla \widetilde{v}_\mu\|_2^2=\|\nabla \widetilde{u}_\mu\|_2^2=S^{\frac{3}{2}}+o_\mu(1)\quad\text{as }\mu\to0.
\end{eqnarray*}
It seems that
$\widetilde{v}_\mu$ will approximate the ground-state level $m(t)=\frac{1}{3}S^{\frac{3}{2}}$ for $N=3$, $\lambda=1$, $q=4$ and $t=t_4^*$ as $\mu\to0$, which suggests that the concentration phenomenon will happen at the ground-state level $m(t)=\frac{1}{3}S^{\frac{3}{2}}$ for $N=3$, $\lambda=1$, $q=4$ and $t=t_4^*$.
\end{remark}

We close this section by the proof of Theorem~\ref{thm0002}.
\medskip

\noindent\textbf{Proof of Theorem~\ref{thm0002}:}\quad It follows from Propositions~\ref{propn0001} and \ref{propn0002}.
\hfill$\Box$

\section{Normalized ground-states for $2<q<2+4/N$}
Let
\begin{eqnarray}\label{eqn0091}
t_q^{**}=\left\{\aligned&0,\quad N\geq4,\\
&t_q^*,\quad N=3,\endaligned\right.
\end{eqnarray}
where $t_q^*$ is given by \eqref{eqn0090}.
Then, by \cite[Theorem~1.2]{ASM12} and Theorem~\ref{thm0002}, \eqref{eq0001} has a ground-state $v_t$ for $t>t_q^{**}$ and $2<q<2+\frac{4}{N}$, which is positive and radial.  By \eqref{eq0005} and \eqref{eq0006}, $(u_t, \lambda_t)$ is a positive normalized solution of \eqref{eqn1001} if and only if
\begin{eqnarray*}
F(t,\mu):=t^{\frac{2}{q\gamma_q-q}-1}-\frac{1-\gamma_q}{a^2\mu^{\frac{2}{q-q\gamma_q}}}\|v_t\|_q^q=0.
\end{eqnarray*}
Clearly, for every $t>t_q^{**}$, there exists a unique
\begin{eqnarray}\label{eq0031}
\mu_t=a^{q\gamma_q-q}\bigg[(1-\gamma_q)\|v_t\|_q^qt^{\frac{q-q\gamma_q+2}{q-q\gamma_q}}\bigg]^{\frac{q-q\gamma_q}{2}}
\end{eqnarray}
such that $F(t,\mu_t)=0$.  Let
\begin{eqnarray*}\label{eq0032}
\overline{\mu}_{q,a}=\sup\{\mu_t>0\mid t>t_q^{**}\}.
\end{eqnarray*}
Then, \eqref{eqn1001} has a positive normalized solution if and only if $\mu<\overline{\mu}_{q,a}$ and $\mu=\mu_t$.
Now, we are prepared for the proof of Theorem~\ref{thm0003}.

\medskip

\noindent\textbf{Proof of Theorem~\ref{thm0003}:}\quad
By \cite[Theorem~1.1]{WW21} and \cite[Theorem~1.6]{JL201}, \eqref{eqn1001} has a normalized ground-state for $\mu>0$ sufficiently small.  Thus, we only need to prove \eqref{eqn1001} has no normalized ground-states for $\mu>0$ sufficiently large, which is equivalent to show that $\overline{\mu}_{q,a}<+\infty$.  Recall that $\gamma_q=\frac{N(q-2)}{2q}$, we always have $q>q\gamma_q$.  It follows from \eqref{eq0031} that $\mu_t\to0$ as $t\to t_q^{**}$ for $N\geq4$ since $t_q^{**}=0$ for $N\geq4$.  For $N=3$, we have $t_q^{**}=t_q^*>0$ and $\|v_t\|_q^q\sim1$ as $t\to t_q^*$ by Proposition~\ref{prop0001}.  Thus, $\mu_t\lesssim1$ as $t\to t_q^{**}$ for all $N\geq3$.  Since $v_t$ is a ground-state of \eqref{eq0001} with the least energy $m(t)$ on the Nehari manifold $\mathcal{N}_t$, by standard arguments (cf. \cite[Lemma~2.2]{CZ12}),
\begin{eqnarray}\label{eq0033}
m(t)=\frac{1}{N}S^{\frac{N}{2}}-\int_{t_q^{**}}^t\frac{1}{q}\|v_\tau\|_q^qd\tau\quad\text{for all $t>t_q^{**}$}
\end{eqnarray}
and
\begin{eqnarray}\label{eq0133}
m'(t)=-\frac{1}{q}\|v_t\|_q^q\quad\text{for a.e. }t>t_q^{**}.
\end{eqnarray}
As that of \eqref{eq0002}, by the Pohozaev identity, we have
\begin{eqnarray}\label{eq0035}
\|\nabla v_t\|_2^2=\gamma_qt\|v_t\|_q^q+\|v_t\|_{2^*}^{2^*}\quad\text{and}\quad \|\nabla v_t\|_2^2=Nm(t).
\end{eqnarray}
Thus, by \eqref{eq0033} and \eqref{eq0133},
\begin{eqnarray*}
Nm(t)+q\gamma_qm'(t)t\geq0\quad\text{ for a.e. }t>t_q^{**},
\end{eqnarray*}
which implies $m(t)t^{\frac{N}{q\gamma_q}}$ is increasing for $t>t_q^{**}$.  Now, let $t_0>t_q^{**}$ with $t_0-t_q^{**}>0$ sufficiently small such that $\mu_t\lesssim1$ for $t<t_0$, then
\begin{eqnarray}\label{eq0034}
m(t)\gtrsim t^{-\frac{N}{q\gamma_q}}\quad\text{for }t\geq t_0.
\end{eqnarray}
On the other hand, by the definition of $t_q^{**}$ given by \eqref{eqn0091}, \cite[Theorem~1.2]{ASM12} and Theorem~\ref{thm0002}, $m(t)<\frac{1}{N}S^{\frac{N}{2}}$ for $t>t_q^{**}$.  Thus,
it is standard to apply the classical elliptic estimates to show that $\|v_t\|_\infty\lesssim1$ for all $t\geq t_0$.  By \eqref{eq0133} and \eqref{eq0035},
\begin{eqnarray*}
Nm(t)=\|\nabla v_t\|_2^2\leqslant(1+O(\frac{1}{t}))\gamma_q\|v_t\|_q^qt=-(1+O(\frac{1}{t}))q\gamma_qm'(t)t\quad\text{ for a.e. }t\geq t_0,
\end{eqnarray*}
which implies that for every $\ve>0$ there exists $t_\ve>0$ such that $m(t)\lesssim t^{-\frac{N}{q\gamma_q+\ve}}$ for $t\geq t_\ve$.  It follows from \eqref{eq0035} once more that
\begin{eqnarray*}
\|v_t\|_q^q\lesssim t^{-\frac{N}{q\gamma_q+\ve}-1}\quad\text{for }t\geq t_\ve.
\end{eqnarray*}
Thus, by \eqref{eq0035} and $\|v_t\|_\infty\lesssim1$ for all $t\geq t_0$, we have
\begin{eqnarray*}
Nm(t)=\|\nabla v_t\|_2^2\leqslant \gamma_q\|v_t\|_q^qt+C_0t^{-\frac{N}{q\gamma_q+\ve}-1} \quad\text{ for }t\geq t_\ve,
\end{eqnarray*}
which implies $m(t)t^{\frac{N}{q\gamma_q}}-C_1t^{-\frac{N}{q\gamma_q+\ve}}$ is decreasing for $t\geq t_\ve$.  Therefore, $m(t)\lesssim t^{-\frac{N}{q\gamma_q}}$ for $t>0$ sufficiently large, which, together with \eqref{eq0034}, implies that
\begin{eqnarray*}
m(t)\sim t^{-\frac{N}{q\gamma_q}}\quad\text{as }t\to+\infty.
\end{eqnarray*}
It follows from \eqref{eq0035} and $\|v_t\|_\infty\lesssim1$ for all $t\geq t_0$ that
\begin{eqnarray*}
\|v_t\|_q^q\sim t^{-\frac{N}{q\gamma_q}-1}\quad\text{as }t\to+\infty.
\end{eqnarray*}
Since
\begin{eqnarray*}
\frac{2}{q-q\gamma_q}-\frac{N}{q\gamma_q}=\frac{2N(q-2-\frac{4}{N})}{(q-2)(2N-q(N-2))}<0 \quad\text{for } 2<q<2+\frac{4}{N},
\end{eqnarray*}
by \eqref{eq0031}, $\overline{\mu}_{q,a}<+\infty$ for $2<q<2+\frac{4}{N}$.
\hfill$\Box$

\section{An application}
In this section, we shall apply our above strategy to study the Schr\"odinger equation~\eqref{eq0066}.
Since there is an additional condition $\|u\|_2^2=r^2$ in \eqref{eq0066}, $\lambda$ in \eqref{eq0066} is not fixed but appears as a Lagrange multiplier.

Let $(u_r, \lambda_r)$ be a solution of \eqref{eq0066}.  Since $V(x)=x_1^2+x_2^2$, we have $\nabla V(x)\cdot x=2V(x)$.  Thus, the Pohozaev identity of \eqref{eq0066} (cf. \cite{BBJV17}) is given by
\begin{eqnarray*}
\frac{1}{6}\|\nabla u_r\|_2^2+\frac{\lambda_rr^2}{2}+\frac{5}{6}\int_{\bbr^3}V(x)u_r^2dx=\frac{1}{p}\|u_r\|_p^p,
\end{eqnarray*}
which, combining the equation~\eqref{eq0066}, implies that
\begin{eqnarray}\label{eq0069}
\lambda_rr^2=\frac{6-p}{2p}\|u_r\|_p^p-2\int_{\bbr^3}V(x)u_r^2dx.
\end{eqnarray}
We define
\begin{eqnarray}\label{eq0068}
w_r(x)=\lambda_r^{-\frac{1}{p-2}}u_r(\lambda_r^{-\frac12}x)\quad\text{and}\quad t_r=\lambda_r
\end{eqnarray}
Then by $V(x)=x_1^2+x_2^2$ and \eqref{eq0069}, $(w_r, t_r)$ is a solution of the following equation:
\begin{eqnarray}\label{eq0067}
\left\{\aligned
&-\Delta w+ w+t^{-2}V(x)w=|w|^{p-2}w\quad\text{in }\mathbb{R}^3,\\
&u\in H^1(\bbr^3),\quad r^2=t^{\frac{10-3p}{2(p-2)}}(\frac{6-p}{2p}\|w\|_p^p-2t^{-2}\int_{\bbr^3}V(x)w^2dx).
\endaligned\right.
\end{eqnarray}
Clearly, if $(w_r, t_r)$ is a solution of \eqref{eq0067}, then, by \eqref{eq0068}, $(u_r, \lambda_r)$ is also a solution of \eqref{eq0066}.

With these basic observations in hands, to find normalized solutions of \eqref{eq0066} with positive Lagrange multipliers, it is equivalent to study the existence of solutions of \eqref{eq0067}.  For this purpose, let us first consider the following equation:
\begin{eqnarray}\label{eq0070}
\left\{\aligned
&-\Delta w+w+t^{-2}V(x)w=|w|^{p-2}w\quad\text{in }\mathbb{R}^3,\\
&w\in H^1(\bbr^3).
\endaligned\right.
\end{eqnarray}
The corresponding functional of \eqref{eq0070} is given by
\begin{eqnarray*}
\mathcal{J}_t(w)=\frac{1}{2}(\|\nabla w\|_2^2+\|w\|_2^2+\int_{\bbr^3}t^{-2}V(x)w^2dx)-\frac{1}{p}\|w\|_p^p.
\end{eqnarray*}
By \cite[Lemma~2.1]{BBJV17} and the Sobolev embedding theorem, this functional is well defined and of class $C^2$ in the Hilbert space
\begin{eqnarray}\label{eq0078}
X=\{w\in H^1(\bbr^3)\mid \int_{\bbr^3}V(x)w^2dx<+\infty\}
\end{eqnarray}
with the norm
\begin{eqnarray*}
\|w\|_X=(\|\nabla w\|_2^2+\int_{\bbr^3}V(x)w^2dx)^{\frac{1}{2}}.
\end{eqnarray*}
We also define the usual Nehari manifold of $\mathcal{J}_t(w)$ as follows:
\begin{eqnarray*}
\mathcal{M}_t=\{w\in X\backslash\{0\}\mid \mathcal{J}_t'(w)w=0\}.
\end{eqnarray*}
The associated fibering map for every $w\not=0$ in $X$ is given by
\begin{eqnarray}\label{eq0076}
J(s)=\frac{s^2}{2}(\|\nabla w\|_2^2+\|w\|_2^2+\int_{\bbr^3}t^{-2}V(x)w^2dx)-\frac{s^p}{p}\|w\|_p^p.
\end{eqnarray}
Since $p>2$, it is standard to show that for every $w\not=0$ in $X$, there exists a unique $s_0'>0$ such that $J(s)$ is strictly increasing for $0<s<s_0'$ and is strictly decreasing for $s>s_0'$.
Let
\begin{eqnarray*}
\mathfrak{m}(t)=\inf_{v\in\mathcal{M}_t}\mathcal{J}_t(v).
\end{eqnarray*}
\begin{definition}
We say that $w$ is a ground-state of \eqref{eq0070} if $w$ is a nontrivial solution of \eqref{eq0070} with $\mathcal{J}_t(w)=\mathfrak{m}(t)$.
\end{definition}
We also need the following equation:
\begin{eqnarray}\label{eq0077}
\left\{\aligned
&-\Delta u+tu+V(x)u=|u|^{p-2}u\quad\text{in }\mathbb{R}^3,\\
&u\in H^1(\bbr^3).
\endaligned\right.
\end{eqnarray}
The corresponding functional of \eqref{eq0077} is given by
\begin{eqnarray*}
\mathcal{I}_t(u)=\frac{1}{2}(\|\nabla u\|_2^2+t\|u\|_2^2+\int_{\bbr^3}V(x)u^2dx)-\frac{1}{p}\|u\|_p^p.
\end{eqnarray*}
This functional is well defined and of class $C^2$ in the Hilbert space $X$, which is given by \eqref{eq0078}.
We define the usual Nehari manifold of $\mathcal{I}_t(u)$ by
\begin{eqnarray*}
\mathcal{P}_t=\{u\in X\backslash\{0\}\mid \mathcal{I}_t'(u)u=0\}.
\end{eqnarray*}
The associated fibering map for every $u\not=0$ in $X$ is given by
\begin{eqnarray}\label{eq0079}
I(s)=\frac{s^2}{2}(\|\nabla u\|_2^2+t\|u\|_2^2+\int_{\bbr^3}V(x)u^2dx)-\frac{s^p}{p}\|u\|_p^p.
\end{eqnarray}
Since $p>2$, it is standard to show that for every $u\not=0$ in $X$, there exists a unique $s_*>0$ such that $I(s)$ is strictly increasing for $0<s<s_*$ and is strictly decreasing for $s>s_*$.
Let
\begin{eqnarray*}
\mathbb{M}(t)=\inf_{v\in\mathcal{P}_t}\mathcal{I}_t(v).
\end{eqnarray*}
\begin{definition}
We say that $u$ is a ground-state of \eqref{eq0077} if $u$ is a nontrivial solution of \eqref{eq0077} with $\mathcal{I}_t(u)=\mathbb{M}(t)$.
\end{definition}

Now, we have the following result of \eqref{eq0070}.
\begin{proposition}\label{prop0002}
Let $\frac{10}{3}<p<6$, then \eqref{eq0070} has a positive ground-state $w_t$ for all $t>0$ satisfying $\|w_t\|_2^2\sim t^{\frac{3p-10}{2(p-2)}}$ as $t\to0$ and $w_t\to w_\infty$ strongly in $H^1(\bbr^3)$ as $t\to+\infty$, where $w_\infty$ is the unique (up to translations) positive solution of the following equation:
\begin{eqnarray}\label{eq0072}
\left\{\aligned
&-\Delta w+w=|w|^{p-2}w\quad\text{in }\mathbb{R}^3,\\
&w\in H^1(\bbr^3).
\endaligned\right.
\end{eqnarray}
Moreover, $w_t$ is unique for $t>0$ sufficiently large.
\end{proposition}
\begin{proof}
The proof is standard so we only sketch it here.  We first prove the existence of ground-states of \eqref{eq0070}.  By the discussion in \cite[4.2 Symmetry of minimizers]{BBJV17}, we know that for the energy level $\mathfrak{m}(t)$, there exists a minimizing sequence $\{w_n\}$ on the Nehari manifold $\mathcal{M}_t$ such that $w_n$ is real and positive.  Moreover, $w_n$ is radial and decreasing w.r.t. $(x_1, x_2)$ for all $x_3$ and $w_n$ is even and decreasing w.r.t. $x_3$ for all $(x_1, x_2)$.  Since $\frac{10}{3}<p<6$, it is standard to use the fibering maps~\eqref{eq0076} to show that $\mathfrak{m}(t)>0$ on $\mathcal{M}_t$.  Thus, by \cite[Lemma~3.4]{BBJV17}, there exists $\{z_n\}\in\bbr$ such that
\begin{eqnarray*}
w_n(x_1,x_2,x_3-z_n)\rightharpoonup w_0\not=0\quad\text{weakly in }X\text{ as }n\to\infty.
\end{eqnarray*}
Since $\frac{10}{3}<p<6$, the fibering map of every $w\not=0$ in $X$, see \eqref{eq0076}, has a unique maximum point $s_0'$ and it interacts the Nehari manifold $\mathcal{M}_t$ only at the unique maximum point $s_0'$.  Thus, we can use standard arguments (cf. \cite[Proposition~3.1]{WW21}) to show that $w_0$ is a positive ground-state of \eqref{eq0070}.
We next prove the convergent conclusion for $t\to+\infty$.
Let $w_t$ be a positive ground-state of \eqref{eq0070} for $t>0$.
Since $V(x)\geq0$, $t>0$ and $w_t$ is positive, we know that $w_t$ satisfies
\begin{eqnarray}\label{eq0071}
-\Delta w_t+w_t\leq w_t^{p-1}\quad\text{in }\bbr^3.
\end{eqnarray}
By using the fibering maps~\eqref{eq0076} in a standard way (cf. \cite[Lemma~3.2]{WW21}), we know that $\mathfrak{m}(t)$ is decreasing w.r.t. $t>0$.  Thus, $\{w_t\}$ is bounded in $H^1(\bbr^3)$.  It follows from \eqref{eq0071} and the classical elliptic estimates that
\begin{eqnarray}\label{eqnew0001}
w_t\lesssim (1+|x|)^{-1}e^{-\frac{1}{2}|x|}\quad\text{in }\bbr^3\text{ for }t\geq1.
\end{eqnarray}
Thus, by $V(x)=x_1^2+x_2^2$,
\begin{eqnarray*}
\int_{\bbr^3}V(x)w_t^2dx\lesssim1\quad\text{for all }t\geq1,
\end{eqnarray*}
which implies that
\begin{eqnarray}\label{eq0097}
t^{-2}\int_{\bbr^3}V(x)w_t^2dx=o_t(1)\quad\text{as }t\to+\infty.
\end{eqnarray}
Now, using the fibering maps~\eqref{eq0076} in a standard way, we know that $\mathfrak{m}(t)\geq \mathfrak{m}+o_t(1)$ as $t\to+\infty$, where
\begin{eqnarray*}
\mathfrak{m}=\inf_{v\in\mathcal{M}}\mathcal{J}(v)
\end{eqnarray*}
with
\begin{eqnarray*}
\mathcal{J}(w)=\frac{1}{2}(\|\nabla w\|_2^2+\|w\|_2^2)-\frac{1}{p}\|w\|_p^p
\end{eqnarray*}
and
\begin{eqnarray*}
\mathcal{M}=\{w\in H^1(\bbr^3)\backslash\{0\}\mid \mathcal{J}'(w)w=0\}.
\end{eqnarray*}
On the other hand, it is well known that \eqref{eq0072} has a unique (up to translations) positive radial solution $w_\infty$, which exponentially decays to zero at infinity.  Thus, using $w_\infty$ as a test function and adapting the property of the fibering maps~\eqref{eq0076} in a standard way, we also have $\mathfrak{m}(t)\leq \mathfrak{m}+o_t(1)$ as $t\to+\infty$.  It follows that $\mathfrak{m}(t)=\mathfrak{m}+o_t(1)$ as $t\to+\infty$, which implies that $\|w_t\|_p^p=\|w_\infty\|_p^p+o_t(1)$.  Now, by standard arguments and the uniqueness of $w_\infty$, we can show that $w_t\to w_\infty$ strongly in $H^1(\bbr^3)$ as $t\to+\infty$.  We now turn to the proof of the convergent conclusion for $t\to0$.  For every $t>0$, let $w_t$ be a positive ground-state of \eqref{eq0070}, then by \eqref{eq0068}, $u_t$ is a positive solution of \eqref{eq0077}.  Moreover, by direct calculations,
\begin{eqnarray*}
\mathcal{J}_t(w_t)=t^{\frac{p-6}{2(p-2)}}\mathcal{I}_t(u_t)\quad\text{and}\quad\mathcal{J}_t'(w_t)w_t=t^{\frac{p-6}{2(p-2)}}\mathcal{I}_t'(u_t)u_t.
\end{eqnarray*}
Thus, $u_t$ is a positive ground-state of \eqref{eq0077} for all $t>0$.  On the other hand, by \cite[Lemma~2.1]{BBJV17}, H\"older and Sobolev inequalities,
\begin{eqnarray*}
\|u\|_p^p\lesssim\|u\|_2^{\frac{6-p}{2}}\|\nabla u\|_2^{\frac{3p-6}{2}}\lesssim\|u\|_X^p\quad\text{for all }u\in X.
\end{eqnarray*}
Thus, by using the fibering maps~\eqref{eq0079} in a standard way, we know that $\mathbb{M}(0)>0$.  By similar arguments as that used above to compare the energy levels $\mathbb{M}(0)$ and $\mathbb{M}(t)$, we can obtain that $\mathbb{M}(t)=\mathbb{M}(0)+o_t(1)$ as $t\to0$.  It follows that $\{u_t\}$ is bounded in $X$ and $\|u_t\|_p^p\sim1$ as $t\to0$.  By \cite[Lemma~2.1]{BBJV17}, $\{u_t\}$ is also bounded in $H^1(\bbr^3)$ as $t\to0$.  Now, by the Lions' lemma (cf. \cite[Lemma I.1]{L84} or \cite[Lemma 1.21]{W96}), we can conclude that $\|u_t\|_2^2\sim1$ as $t\to0$.  It follows from \eqref{eq0068} that $\|w_t\|_2^2\sim t^{\frac{3p-10}{2(p-2)}}$ as $t\to0$.  We close this proof by showing the uniqueness of $w_t$ for $t>0$ sufficiently large.  Let $w_t$ and $w_t'$ be two different positive ground-states of \eqref{eq0070} and we define $\phi_t=\frac{w_t-w_t'}{\|w_t-w_t'\|_{L^\infty}(\bbr^3)}$.  Then by the Taylor expansion,
\begin{eqnarray*}
-\Delta\phi_t+\phi_t+t^{-2}V(x)\phi_t=(p-1)(w_t+\theta(w_t-w_t')))^{p-2}\phi_t,\quad\text{in }\bbr^3,
\end{eqnarray*}
where $\theta\in(0, 1)$.
Since $V(x)\geq0$, by \eqref{eqnew0001},
\begin{eqnarray*}
-\Delta(\phi_t)^2+\frac{3}{2}(\phi_t)^2\leq0,\quad\text{in }\bbr^3.
\end{eqnarray*}
Thus, by the maximum principle, $|\phi_t|\lesssim e^{-\frac{1}{2}|x|}$ for $|x|\geq1$.
It is standard to show that $\phi_t\to\phi$ strongly in any compact sets as $t\to+\infty$ and
\begin{eqnarray*}
-\Delta\phi+\phi=(p-1)w_\infty^{p-2}\phi,\quad\text{in }\bbr^3.
\end{eqnarray*}
Note that $w_t$ and $w_t'$ are radial w.r.t. $(x_1, x_2)$ for all $x_3$ and even w.r.t. $x_3$ for all $(x_1, x_2)$.  Thus, $\phi_t$ is also radial w.r.t. $(x_1, x_2)$ for all $x_3$ and even w.r.t. $x_3$ for all $(x_1, x_2)$.
Now, by the well-known nondegeneracy of $w_\infty$, we have $\phi_\infty\equiv0$.  It, together with $|\phi_t|\lesssim e^{-\frac{1}{2}|x|}$ for $|x|\geq1$, contradicts $\|\phi_t\|_{L^{\infty}(\bbr^3)}=1$.  Therefore, $w_t$ is unique for $t>0$ sufficiently large.
\end{proof}

Let $w_t$ be a positive ground-state of \eqref{eq0070} given by Proposition~\ref{prop0002} and we define
\begin{eqnarray*}
f(r,t):=r^2-t^{\frac{10-3p}{2(p-2)}}(\frac{6-p}{2p}\|w_t\|_p^p-2t^{-2}\int_{\bbr^3}V(x)w_t^2dx).
\end{eqnarray*}
By Proposition~\ref{prop0002}, for every $t>0$ sufficiently large, there exists a unique
\begin{eqnarray}\label{eq0081}
r_t=(t^{\frac{10-3p}{2(p-2)}}(\frac{6-p}{2p}\|w_t\|_p^p-2t^{-2}\int_{\bbr^3}V(x)w_t^2dx))^{\frac{1}{2}}>0
\end{eqnarray}
such that $f(r_t,t)=0$.  Thus,
by \eqref{eq0068}, $(u_{r_t}, t)$ is a positive normalized solution of \eqref{eq0066} with a positive Lagrange multiplier $t>0$.
We are now prepared for the proof of Theorem~\ref{thm0004}.

\medskip

\noindent\textbf{Proof of Theorem~\ref{thm0004}:}\quad
By the uniqueness of $w_t$ given by Proposition~\ref{prop0002} for $t>0$ sufficiently large, say $t>T_*$.  It is standard to show that $\int_{\bbr^3}V(x)w_t^2$ is continuous for $t>T_*$.  Note that
by Proposition~\ref{prop0002},
\begin{eqnarray*}
(\frac{6-p}{2p}\|w_t\|_p^p-2t^{-2}\int_{\bbr^3}V(x)w_t^2dx)=\frac{6-p}{2p}\|w_\infty\|_p^p+o_t(1).
\end{eqnarray*}
Thus, by $\frac{10}{3}<p<6$, for every $r<(T_*^{\frac{10-3p}{2(p-2)}}(\frac{6-p}{2p}\|w_{T_*}\|_p^p-2T_*^{-2}\int_{\bbr^3}V(x)w_{T_*}^2dx))^{\frac{1}{2}}$, $f(r,t)=0$ has a solution $t_r>T_*$.
This, together with \cite[Theorem~2]{BBJV17}, implies that \eqref{eq0066} has a second positive normalized solution $u_{r,2}$ with a positive Lagrange multiplier $\lambda_{r,2}$.  The asymptotic behavior of $u_{r,2}$ and $\lambda_{r,2}$ is obtained by \eqref{eq0068} and \eqref{eq0081}.  It remains to show that $u_{r,2}$ is a mountain-pass solution of \eqref{eq0066} for $r>0$ sufficiently small.  As that in \cite[Remark~1.10]{BBJV17}, we introduce the mountain-pass level
\begin{eqnarray*}
\alpha(r)=\inf_{g\in\Gamma_r}\max_{t\in[0, 1]}\mathcal{Y}(g[t]),
\end{eqnarray*}
where $\mathcal{Y}(u)=\frac{1}{2}\|u\|_{X}^2-\frac{1}{p}\|u\|_p^p$
and
\begin{eqnarray*}
\Gamma_r=\{g[s]\in C([0, 1], \mathcal{S}_r)\mid g[0]=u_{r,1}\quad\text{and}\quad \mathcal{Y}(g[1])<\mathcal{Y}(g[0])\}
\end{eqnarray*}
with $u_{r,1}$ being a local minimizer of $\mathcal{Y}(u)$ in $\mathcal{S}_r$ found in \cite{BBJV17} and $\mathcal{S}_r=\{u\in X\mid \|u\|_2^2=r^2\}$.  Let
\begin{eqnarray*}
B_{\rho,X,t}=\{u\in X\mid \|u\|_{X,t}^2\leq \rho^2\},
\end{eqnarray*}
where $\|u\|_{X,t}$ is a norm in $X$ given by
\begin{eqnarray*}
\|w\|_{X,t}=(\|\nabla w\|_2^2+\|w\|_2^2+t^{-2}\int_{\bbr^3}V(x)w^2dx)^{\frac{1}{2}}.
\end{eqnarray*}
Then by \cite[Lemma~2.1]{BBJV17} and the Sobolev inequality, for a fixed $\rho>0$ sufficiently small, it can be proved by using $\frac{10}{3}<p<6$ in a standard way that
\begin{eqnarray*}
\mathfrak{m}(t)=\inf_{h\in\Theta}\max_{t\in[0, 1]}\mathcal{J}_t(h[s]),
\end{eqnarray*}
where
\begin{eqnarray*}
\Theta=\{h[t]\in C([0, 1], X)\mid h[0]\in B_{\rho,X,t}\quad\text{and}\quad \mathcal{J}_t(h[1])<\frac{1}{4}\rho^2\}.
\end{eqnarray*}
Now, for every $g[s]\in\Gamma_r$, we define $g^*[s]=\lambda_{r,2}^{-\frac{1}{p-2}}g[s](\lambda_{r,2}^{-\frac{1}{2}}x)$.  Then
\begin{eqnarray*}
\mathcal{J}_{\lambda_{r,2}}(g^*[s])=\lambda_{r,2}^{\frac{p-6}{2(p-2)}}(\mathcal{Y}(g[s])+\frac{\lambda_{r,2}r^2}{2}).
\end{eqnarray*}
By \cite[Theorem~1]{BBJV17} and \eqref{eq0099},
\begin{eqnarray*}
\|g^*[0]\|_{X,\lambda_{r,2}}^2\lesssim r^2\lambda_{r,2}^{\frac{p-6}{2(p-2)}}\sim\lambda_{r,2}^{-1}\to0\quad\text{as }r\to0.
\end{eqnarray*}
Thus, $g^*[0]\in B_{\rho,X,\lambda_{r,2}}$ for $r>0$ sufficiently small and $\mathcal{J}_t(g^*[0])\to0$ as $r\to0$.  By the definition of $g[t]$, we also have $\mathcal{J}_t(g^*[1])<\frac{1}{4}\rho^2$.  It follows that $g^*[t]\in\Theta$, which implies
\begin{eqnarray*}
\mathfrak{m}(\lambda_{r,2})\leq\lambda_{r,2}^{\frac{p-6}{2(p-2)}}(\alpha(r)+\frac{\lambda_{r,2}r^2}{2}).
\end{eqnarray*}
On the other hand, the fibering map of $\mathcal{Y}(u)$ at $u_{r,2}$ is given by
\begin{eqnarray*}
\mathcal{T}_{u_{r,2}}(\tau)=\frac{\tau^2}{2}\|\nabla u_{r,2}\|_2^2+\frac{1}{2\tau^2}\int_{\bbr^3}V(x)u_{r,2}^2dx-\frac{\tau^{p\gamma_p}}{p}\|u_{r,2}\|_p^p.
\end{eqnarray*}
By direct calculations,
\begin{eqnarray*}
\mathcal{T}_{u_{r,2}}'(\tau)=\tau\|\nabla u_{r,2}\|_2^2-\frac{1}{\tau^3}\int_{\bbr^3}V(x)u_{r,2}^2dx-\gamma_p\tau^{p\gamma_p-1}\|u_{r,2}\|_p^p
\end{eqnarray*}
and
\begin{eqnarray*}
\mathcal{T}_{u_{r,2}}''(\tau)=\|\nabla u_{r,2}\|_2^2+\frac{3}{\tau^4}\int_{\bbr^3}V(x)u_{r,2}^2dx-\gamma_p(p\gamma_p-1)\tau^{p\gamma_p-2}\|u_{r,2}\|_p^p.
\end{eqnarray*}
Clearly, $\mathcal{T}_{u_{r,2}}'(1)=0$.  Moreover, by \eqref{eq0098}, \eqref{eq0097} and the Pohozaev identity of $w_\infty$,
\begin{eqnarray*}
\mathcal{T}_{u_{r,2}}''(1)=\lambda_{r,2}^{\frac{6-p}{2(p-2)}}(\gamma_p\|w_\infty\|_p^p(2-p\gamma_p)+o_r(1))<0
\end{eqnarray*}
for $r>0$ sufficiently small.  Now, let $h(\tau)=\tau^4\|\nabla u_{r,2}\|_2^2-\gamma_p\tau^{p\gamma_p+2}\|u_{r,2}\|_p^p,$
then,
\begin{eqnarray*}
\max_{\tau\geq0}h(\tau)=\bigg[\frac{4\|\nabla u_{r,2}\|_2^2}{\gamma_p(p\gamma_p+2)\|u_{r,2}\|_p^p}\bigg]^{\frac{4}{p\gamma_p-2}}\frac{p\gamma_p-2}{p\gamma_p+2}\|\nabla u_{r,2}\|_2^2>\int_{\bbr^3}V(x)u_{r,2}^2dx.
\end{eqnarray*}
It follows that there exists $\tau_r<1$ such that $\mathcal{T}_{u_{r,2}}'(\tau_r)=0$ and $\mathcal{T}_{u_{r,2}}''(\tau_r)>0$.  We claim that $\tau_r\to0$ as $r\to0$.  If not, then, there exists $r_n\to0$ such that $\tau_{r_n}\gtrsim1$ as $n\to\infty$.  Without loss of generality, we may assume $\tau_r\gtrsim1$ for all $r>0$ sufficiently small.  By $\mathcal{T}_{u_{r,2}}'(\tau_{r})=0$, \eqref{eq0098}, \eqref{eq0097} and the Pohozaev identity of $w_\infty$,
\begin{eqnarray*}
\mathcal{T}_{u_{r,2}}''(\tau)&=&\lambda_{r,2}^{\frac{6-p}{2(p-2)}}(4\|\nabla w_\infty\|_2^2-\gamma_p(p\gamma_p+2)\|w_\infty\|_p^p\tau_r^{p\gamma_p-2}+o_r(1))\\
&=&\lambda_{r,2}^{\frac{6-p}{2(p-2)}}\gamma_p(4-(p\gamma_p+2)\tau_r^{p\gamma_p-2}+o_r(1))\|w_\infty\|_p^p,
\end{eqnarray*}
which implies $\tau_r<(\frac{4}{p\gamma_p+2})^{\frac{1}{p\gamma_p-2}}<1$.  Without loss of generality, we may assume that $\tau_r\to\tau_0$ as $r\to0$.  Then by $\mathcal{T}_{u_{r,2}}'(\tau_r)=0$, \eqref{eq0097} and the fact that $w_\infty$ solves \eqref{eq0096}, we must have $\tau_0=0$.  It is impossible.  Thus, we must have $\tau_r\to0$ as $r\to0$.  By \eqref{eq0099} and \eqref{eq0098},
\begin{eqnarray*}
\frac{1}{\tau_r^4\lambda_{r,2}^2}(\int_{\bbr^3}V(x)w_\infty^2dx+o_r(1))=\|\nabla w_\infty\|_2^2+o_r(1).
\end{eqnarray*}
It follows from \eqref{eq0099} and \eqref{eq0098} that
\begin{eqnarray*}
\|(u_{r,2})_{\tau_r}\|_X^2&=&\tau_r^2\|\nabla u_{r,2}\|_2^2+\tau_r^{-2}\int_{\bbr^3}V(x)u_{r,2}^2dx\sim\lambda_{r,2}^{\frac{10-3p}{2(p-2)}}\sim r^2
\end{eqnarray*}
as $r\to0$, where $(u_{r,2})_{\tau_r}=\tau_r^{\frac{3}{2}}u_{r,2}(\tau_r x)$.  Thus, $(u_{r,2})_{\tau_r}\in B_{r\chi,X,1}$ for a fixed and large $\chi>0$.  Since $B_{r\chi,X,1}$ is connected, we can find a continuous path $\Upsilon: [0, 1]$ with $\Upsilon(0)=u_{r,1}$ and $\Upsilon(1)=(u_{r,2})_{\tau_r}$.  Now, we define
\begin{eqnarray*}
h^{**}[s]=\left\{\aligned&\Upsilon[(2s)],\quad 0\leq s\leq\frac12,\\
&(u_{r,2})_{2(1-s)\tau_r+(2s-1)\tau_{r,*}}, \quad\frac12\leq s\leq1,\\
\endaligned\right.
\end{eqnarray*}
where we choose $\tau_{r,*}>1$ such that $\mathcal{T}_{u_{r,2}}(\tau_{r,*})<\mathcal{Y}(u_{r,1})$.  Note that $\mathcal{Y}(u)\lesssim r^2$ in $B_{r\chi,X,1}$ and $\mathcal{Y}(u_{r,2})\gtrsim1$ by \eqref{eq0099} and \eqref{eq0098}. Thus, for $r>0$ sufficiently small, $h^{**}[s]\in\Gamma_r$ and
\begin{eqnarray*}
\alpha(r)\leq\max_{0\leq s\leq1}h^{**}[s]=\mathcal{T}_{u_{r,2}}(1)=\mathfrak{m}(\lambda_{r,2})\lambda_{r,2}^{\frac{6-p}{2(p-2)}}-\frac{\lambda_{r,2}r^2}{2}.
\end{eqnarray*}
Therefore, $\mathfrak{m}(\lambda_{r,2})\lambda_{r,2}^{\frac{6-p}{2(p-2)}}-\frac{\lambda_{r,2}r^2}{2}=\alpha(r)$ and $u_{r,2}$ is a mountain-pass solution of \eqref{eq0066} for $r>0$ sufficiently small.
\hfill$\Box$

\section{Acknowledgements}
The research of J. Wei is
partially supported by NSERC of Canada and the research of Y. Wu is supported by NSFC (No. 11971339, 12171470).


\begin{thebibliography}{100}
\bibitem{AIK20}
T. Akahori, S. Ibrahim, H. Kikuchi, Linear instability and nondegeneracy of ground state for combined
power-type nonlinear scalar field equations with the Sobolev critical exponent and large frequency parameter, {\it Proc. Roy. Soc. Edinburgh Sect. A,} {\bf150} (2020), 2417--2441.

\bibitem{AIKN12}
T. Akahori, S. Ibrahim, H. Kikuchi and H. Nawa, Existence of a ground state and blow-up problem for a
nonlinear Schr\"odinger equation with critical growth, {\it Differ. Integral Equ.,} {\bf25} (2012), 383--402.

\bibitem{AIKN2021}
T. Akahori, S. Ibrahim, H. Kikuchi, H. Nawa, Global dynamics above the ground state energy for
the combined power type nonlinear Schrodinger equations with energy critical growth at low frequencies, arXiv:1510.08034 [Math. AP](to appear in Memoirs of the AMS).


\bibitem{AIIKN19}
T. Akahori, S. Ibrahim, N. Ikoma, H. Kikuchi, H. Nawa, Uniqueness and nondegeneracy of ground
states to nonlinear scalar field equations involving the Sobolev critical exponent in their nonlinearities for
high frequencies, {\it Calc. Var. PDEs,} {\bf58} (2019), 120.

\bibitem{AP86}
F. Atkinson, L. Peletier, Emden-Fowler equations involving critical exponents, {\it Nonlinear Anal.,} {\bf10}(1986), 755--776.

\bibitem{ASM12}
C.O. Alves, M.A.S. Souto, M. Montenegro, Existence of a ground state solution  for a nonlinear
scalar field equation with critical growth, {\it Calc. Var. PDEs,} {\bf43}(2012), 537--554.

\bibitem{BBJV17}
J. Bellazzini, N. Boussaid, L. Jeanjean, N. Visciglia, Existence and Stability of Standing Waves for Supercritical
NLS with a Partial Confinement, {\it Commun. Math. Phys.,} {\bf 353}(2017), 229--251.

\bibitem{BL83}
H. Berestycki, P.-L. Lions, Nonlinear scalar field equations. I: Existence of a ground state,
{\it Arch. Ration. Mech. Anal.,} {\bf82}(1983), 313--345.

\bibitem{BN83}
H. Br\'ezis, L. Nirenberg, Positive solutions of nonlinear elliptic equations involving critical Sobolev
exponents. {\it Commun. Pure Appl. Math.,} {\bf36}(1983), 437--477.

\bibitem{CDG16}
W. Chen, J. D\'avila, I. Guerra, Bubble tower solutions for a supercritical elliptic problem in $\bbr^N$, {\it Ann. Sc.
Norm. Super. Pisa Cl. Sci. (5),} {\bf15} (2016), 85--116.

\bibitem{CZ12}
Z. Chen, W. Zou, On the Brez\'is-Nirenberg problem in a ball,  {\it Differ. Integral Equ.,} {\bf25}(2012), 527--542.

\bibitem{C72}
C. Coffman, Uniqueness of the ground state solution for $\Delta u-u+u^3=0$ and a variational characterization of other solutions. {\em Arch. Rational Mech. Anal.,} {\bf46} (1972),  81--95.

\bibitem{CG20}
M. Coles, S. Gustafson, Solitary Waves and Dynamics for Subcritical Perturbations of Energy Critical NLS, {\it Publ. Res. Inst. Math. Sci.,} {\bf56} (2020), 647--699.

\bibitem{DPG13}
J. D\'avila, M. del Pino, I. Guerra, Non-uniqueness of positive ground states of non-linear Schr\"odinger equations, {\it Proc. Lond. Math. Soc.,} {\bf106}(2013), 318--344.

\bibitem{FG03}
A. Ferrero, F. Gazzola, On subcriticality assumptions for the existence of ground states of quasilinear elliptic equations, {\it Adv. Differential Equations,} {\bf8} (2003), 1081--1106.

\bibitem{GNN81}
B. Gidas, W. Ni and L. Nirenberg, Symmetry of positive
solutions of nonlinear elliptic equations in $\mathbb{R}^N$, in: L. Nachbin
(Ed.), Math. Anal. Appl. Part A, Advances in Math. Suppl. Studies, vol. 7A,
Academic Press, 1981, pp. 369--402.

\bibitem{GS03}
F. Gazzola, J. Serrin, Asymptotic behavior of ground states of quasilinear elliptic problems with two vanishing parameters,  {\it Ann. Inst. H. Poincar\'e Anal. Non Lin\'eaire,} {\bf19}(2002), 477--504.

\bibitem{JL201}
L. Jeanjean, T. Le, Multiple normalized solutions for a Sobolev critical Schr\"odinger equation, {\it Math. Ann.}, DOI: 10.1007/s00208-021-02228-0.

\bibitem{KP89}
M. Knaap, L. Peletier, Quasilinear elliptic equations with nearly critical growth, {\it Comm. PDEs,} {\bf14}(1989), 1351--1383.

\bibitem{K89}
M. Kwong, Uniqueness of positive solution of $\Delta u-u+u^p=0$ in $\mathbf{R}^N$. {\em Arch. Rational Mech. Anal.,} {\bf105} (1989),  243--266.

\bibitem{L84}
P.L. Lions, The concentration-compactness principle in the calculus of variations. The locally compact case, part 2, {\it  Ann. Inst. H. Poincar\'e Anal. Non Lin\'eaire,} {\bf1} (1984), 223-283.

\bibitem{LLT17}
J. Liu, J.-F. Liao, C.-L. Tang, Ground state solution for a class of Schr\"odinger equations involving
general critical growth term, {\it Nonlinearity,} {\bf30} (2017), 899--911.

\bibitem{MM21}
S. Ma, V. Moroz, Asymptotic profiles for a nonlinear Schr\"odinger equation with critical combined powers nonlinearty, arXiv2108.01421 [Math. AP].

\bibitem{M93}
K. McLeod, Uniqueness of positive radial solutions of $\Delta u+f(u)=0$ in $\bbr^N$. II, {\it Trans. Amer. Math. Soc.,} {\bf339} (1993), 495--505.

\bibitem{MS87}
K. McLeod, J. Serrin, Uniqueness of positive radial solutions of $\Delta u+f(u)=0$ in $\bbr^N$,  {\it Arch. Ration. Mech. Anal.,} {\bf99} (1987), 115--145.

\bibitem{PS83}
L. Peletier, J. Serrin, Uniqueness of positive solutions of semilinear equations in $\bbr^N$, {\it Arch. Ration. Mech. Anal.,} {\bf81} (1983), 181--197.

\bibitem{PS86}
L. Peletier, J. Serrin, Uniqueness of nonnegative solutions of semilinear equations in $\bbr^N$, {\it J. Differential Equations,} {\bf61} (1986), 380--397.

\bibitem{PS98}
P. Pucci, J. Serrin, Uniqueness of ground states for quasilinear elliptic operators, {\it Indiana Univ. Math. J.,} {\bf47} (1998), 501--528.

\bibitem{ST00}
J. Serrin, M. Tang, Uniqueness of ground states for quasilinear elliptic equations, {\it Indiana Univ. Math. J.,} {\bf49} (2000), 897--923.

\bibitem{S20}
N. Soave, Normalized ground states for the NLS equation with combined nonlinearities: The Sobolev critical case. {\it J. Funct. Anal.,} {\bf 279}(2020), 108610.

\bibitem{WW21}
J. Wei, Y. Wu, Normalized solutions for Schrodinger equations with critical Sobolev exponent and mixed nonlinearities, arXiv:2102.04030 [Math. AP].

\bibitem{W96}
M. Willem, Minimax Theorems. Birkh\"auser, Boston, 1996.

\bibitem{ZZ12}
J. Zhang, W. Zou, A Berestycki-Lions theorem revisited, {\it Commun. Contemp. Math.,} {\bf14} (2012), 1250033.
\end{thebibliography}
\end{document}